%% file: MainPaperFile.tex
\documentclass{interact}
\usepackage{url,hyperref}
\usepackage[english]{babel}
\usepackage{algorithm}
\usepackage{algpseudocodex}
\algrenewcommand\algorithmicrequire{\textbf{Input:}\phantom{\textbf{O}}}
\algrenewcommand\algorithmicensure{\textbf{Output:}}
\usepackage[T1]{fontenc}
\usepackage{amssymb,amsmath,amsfonts,amsthm}
\usepackage[shortlabels]{enumitem}
\usepackage{comment}

\usepackage[numbers,sort&compress]{natbib}% Citation support using natbib.sty
\bibpunct[, ]{[}{]}{,}{n}{,}{,}% Citation support using natbib.sty
\makeatletter% @ becomes a letter
\def\NAT@def@citea{\def\@citea{\NAT@separator}}% Suppress spaces between citations using natbib.sty
\makeatother% @ becomes a symbol again
\usepackage{todonotes}

\theoremstyle{plain}% Theorem-like structures provided by amsthm.sty
\newtheorem{theorem}{Theorem}[section]
\newtheorem{lemma}[theorem]{Lemma}
\newtheorem{corollary}[theorem]{Corollary}
\newtheorem{proposition}[theorem]{Proposition}

\theoremstyle{definition}
\newtheorem{definition}[theorem]{Definition}
\newtheorem{example}[theorem]{Example}

\theoremstyle{remark}
\newtheorem{remark}[theorem]{Remark}

\newcommand{\QQ}{\mathbb{Q}}

\newcommand{\I}{{\mathcal I}}
\newcommand{\K}{{\mathcal K}}

\DeclareMathOperator{\cls}{cls}

\DeclareMathOperator{\lt}{lt}
\DeclareMathOperator{\LC}{{\rm lc}}
\DeclareMathOperator{\LM}{{\rm lm}}
\DeclareMathOperator{\LT}{{\rm lt}}

\def\I{\mathcal{I}}

\begin{document}
\title{Deterministic Determination of Axial Constants, Sectional
  Regularities and Annihilator Numbers}
\author{\name{Amir Hashemi\textsuperscript{a} and Rozhin
  Sadr\textsuperscript{a} and Werner
  M. Seiler\textsuperscript{b}\thanks{Email: Amir.Hashemi@iut.ac.ir,
    Rozhinsadr1385@gmail.com, seiler@mathematik.uni-kassel.de}}
 \affil{\textsuperscript{a}Department of Mathematical Sciences, Isfahan University of
    Technology, Isfahan 84156-83111, Iran\\  \textsuperscript{b}Institut f\"{u}r Mathematik, Universit\"{a}t Kassel,
    34109 Kassel, Germany}}
\maketitle    
  \begin{abstract}
    Axial constants and sectional regularities are recently introduced
    invariants of homogeneous polynomial ideals that, like the already
    earlier introduced annihilator numbers, depend on generic choices of
    coordinates.  We provide effective characterisations of the genericity
    underlying these invariants in terms of $\beta$-maximal position and
    strong stability. In particular, we show that $\beta$-maximal position
    suffices for determining generic annihilator numbers and sectional
    regularities, while its combination with strong stability yields
    corresponding results for axial constants. We also develop
    deterministic linear algebra procedures for detecting the required
    generic positions and for computing these invariants.
\end{abstract}
\begin{amscode}
  13P10; 13D02; 68W30
\end{amscode}
%%%%%%%%%%%%%%%%%%%%%%%%%%%%%%%%%%
\section{Introduction}

The Castelnuovo-Mumford regularity, introduced by Mumford
\cite{zbMATH02195682}, is one of the most fundamental homological
invariants in commutative algebra and algebraic geometry.  It plays a
central role in bounding the degrees of syzygies, controlling the
complexity of Gr\"obner basis computations, and understanding the geometry
of projective varieties.  Since the pioneering works of Galligo
\cite{zbMATH03842031} and Bayer and Stillman \cite{zbMATH04014856}, a wide
variety of approaches have been developed to study and compute regularity,
ranging from Gr\"obner bases~\cite{bermejo} to involutive
bases~\cite{zbMATH05636280,SP} (see also \cite{S}).  In this paper, we
focus on two related notions recently introduced in the literature, namely
\emph{axial constants}~\cite{CH25a} and \emph{sectional
  regularities}~\cite{axial}.  As shown in the latter article, they are
closely related to the \emph{annihilator numbers} introduced by
Trung~\cite{zbMATH04043980} already much earlier (see also \cite{herzog}).

All these invariants depend on a choice of linear forms or coordinates, but
generic choices yield always the same values which can also be related to
the generic initial ideal $\mathrm{Gin}(\mathcal{I})$ of $\mathcal{I}$ with
respect to the degree reverse lexicographic order (introduced by Galligo
\cite{zbMATH03465689} in characteristic zero and by Bayer and Stillman
\cite{bs:reverse} in arbitrary characteristic).  Let
$\mathcal{R}=\mathcal{K}[x_1,\ldots,x_n]$ be a polynomial ring over an
infinite field $\mathcal{K}$ and $\mathcal{I}\subset\mathcal{R}$ a
homogeneous ideal. The $i$-th generic sectional regularity of $\mathcal{I}$
is
\( \operatorname{reg}\big(\mathrm{Gin}(\mathcal{I})+\langle
x_1,\ldots,x_{n-i}\rangle\big) \).  Similarly, the $i$-th generic axial
constant of $\mathcal{I}$ is given by
\( \min{ \bigl\{ j \in \mathbb{N} \mid x_i^j \in \mathrm{Gin}(\mathcal{I})
  \bigr\}} \). The generic annihilator numbers are obtained from
annihilator modules associated with a generic sequence of linear forms or,
equivalently, from the corresponding annihilator numbers of
$\mathcal R/\operatorname{Gin}(\mathcal I)$ with respect to the coordinate
variables.  Precise definitions of these concepts also for the non-generic
case will be given below.

While working with the generic initial ideal is nice from a theoretical
point of view, as one can then always assume that one is in a generic
position, the situation is quite different when one studies concrete
examples which typically are not in a generic position.  The main purpose
of this work is to provide effective characterisations of the generic
positions underlying these invariants and, consequently, to develop
deterministic methods for their computation using the theory of Pommaret
bases.  Rather than relying on a generic change of coordinates, we identify
explicit algebraic conditions that allow us to effectively verify whether a
given coordinate system is in the desired generic position.

A central role is played by the notion of $\beta$-maximal position,
introduced in \cite{HSS} for computing generic annihilator numbers.  We
show that $\beta$-maximal position also suffices for determining the
generic sectional regularities from a Pommaret basis. When strong stability
is present in addition, these sectional regularities are related directly
to the generic axial constants. We furthermore develop a deterministic
linear algebra procedure for determining coordinate changes that put an
ideal into GIN position, assuming that its Castelnuovo--Mumford regularity
is known.

The structure of the paper is as follows.  Section~\ref{sec:2} mainly
recalls the necessary background on Gr\"obner bases, Pommaret bases, and
quasi-regular sequences, but also provides some new results on annihilator
numbers.  Section~\ref{sec:3} discusses strongly stable ideals and their
regularity properties.  While most results in this section are known and
have been established in \cite{axial}, we provide new proofs via Pommaret
bases.  Section~\ref{sec:5} reviews $\beta$-maximal position and its
relevant properties and develops a deterministic linear algebra procedure
for determining coordinate changes that put an ideal into GIN position.
Finally, Section~\ref{sec:6} studies the sectional regularities in
$\beta$-maximal position and, under additional strong-stability
assumptions, relates them to the axial constants.

\input{Preliminaries}
\input{StronglyStable}
\input{B-Max}
\input{Sreg}

\section{Conclusions}

As in \cite{axial}, we studied three different families of generic
invariants of homogeneous ideals: annihilator numbers, sectional
regularities and axial constants.  The main difference of our work compared
to \cite{axial} is the consequent use of Pommaret bases also as a
theoretical tool for the proofs, as it makes all results immediately
constructive and avoids the use of the generic initial ideal.  Of course,
one can argue that Pommaret bases are also a generic notion, as they only
exist in quasi-stable position which is a generic position.  But the real
problem with the GIN position is that it is very hard to verify
effectively.  To the best of our knowledge, no computer algebra system
gives a guarantee that it computes the correct generic initial ideal, they
all just rely on the fact that a random transformation is \emph{almost}
always the right one.  By contrast, just the existence of a Pommaret basis
is a certificate for quasi-stable position and we have deterministic
algorithms for its construction which in many cases lead to rather sparse
transformations, whereas random transformations are always dense.

As the authors of \cite{axial} already remarked, the sectional regularities
are better behaved invariants than the axial constants, as the latter
become infinite for index values lower than the dimension.  The sectional
regularities become constant only for index values lower than the depth.
Thus for ideals which are not Cohen-Macaulay there is an index range where
the sectional regularities still carry information but not the axial
constants.  But it is also clear that the finest invariants are the
annihilator numbers, as the sectional regularities only describe parts of
the boundary of their diagram.  We have here a similar situation as with
the classical regularity which together with the projective dimension
provides a ``bounding box'' for the Betti diagram but contains of course
less information than the full Betti diagram.  We recalled from \cite{HSS}
that the annihilator numbers can be immediately read off from a Pommaret
basis, as they encode how the generators distribute over degrees and
classes.  Unfortunately, one obtains the generic annihilator numbers only
if the ideal is in $\beta$-maximal position whereas for example the
Castelnuovo-Mumford regularity can be read off from \emph{any} Pommaret
basis.

From a computational point of view, our most important result is probably
Algorithm~\ref{alg:bmax}.  Despite its high complexity, we conjecture that
it is the most efficient algorithm presented so far that \emph{guarantees}
to yield the correct generic initial ideal.  The discussion around it also
shows that while theoretically $\beta$-maximal position is a more relaxed
notion of genericity than GIN position, it is computationally just as
difficult.  In fact, with our algorithm it is easier to go directly for GIN
position than for a different $\beta$-maximal position (provided it
exists).  We remark that it is shown in \cite{HSS} that in GIN position a
reduced Gr\"obner basis is simultaneously the unique Pommaret basis so that
one does not need an implementation of an involutive algorithm; by
contrast, this is not necessarily the case in $\beta$-maximal position, as
the latter one does not imply stable position.  It is well-known that only
for stable ideals, the minimal basis is also the Pommaret basis
\cite{mall:pom}.

\section*{Acknowledgements}
The research of the second author was partially conducted
during an extended stay at Kassel University funded by Erasmus+ (KA 171).

%%%%%%%%%%%%%%%%%%%%%%%%%%%%%%
\bibliography{Ref}
\bibliographystyle{tfnlm}
\end{document}

%% file: Preliminaries.tex
\section{Quasi-Stability and Pommaret Bases}
\label{sec:2}

Let $\mathcal{R} = \K[x_1, \dots, x_n]$ be a polynomial ring in finitely
many variables over an infinite field $\K$ and let
\( \mathcal{I} = \langle f_1, \dots, f_k \rangle \subset \mathcal{R} \) be
the ideal generated by homogeneous polynomials
$f_1, \dots, f_k \in \mathcal{R}$.  Given a finite set
$F\subset\mathcal{R}$ of polynomials, we denote by $\deg{(F)}$ the maximal
degree of a polynomial $f\in F$. A power product
\( x_1^{\alpha_1} \cdots x_n^{\alpha_n} \) is called a \emph{term} and is
denoted by $x^\alpha$ with the exponent vector
$\alpha = (\alpha_1, \dots, \alpha_n)$.  We fix a term ordering $\prec$ on
$\mathcal{R}$. We will usually use the degree reverse lexicographic order,
denoted $\prec_{drl}$, which is defined as follows: for two terms
$x^\alpha$ and $x^\beta$, we write \( x^\alpha \prec_{drl} x^\beta \) if
\(\sum_{i=1}^n \alpha_i < \sum_{i=1}^n \beta_i\) or, if
\(\sum_{i=1}^n \alpha_i = \sum_{i=1}^n \beta_i\), we have
$\beta_i < \alpha_i$ for the smallest index $i$ such that
$\alpha_i \neq \beta_i$.  With this definition, we have
$x_1 \prec_{drl} \cdots \prec_{drl} x_n$. For a fixed term ordering $\prec$
and any non-zero polynomial
\( f = c_\alpha x^\alpha + c_\beta x^\beta + \cdots + c_\gamma x^\gamma \),
with $0 \ne c_\alpha \in \K$ and
\( x^\gamma \prec \cdots \prec x^\beta \prec x^\alpha \), the \emph{leading
  term} of $f$ is defined as $\LT(f) = x^\alpha$. The coefficient
$c_\alpha$ is called the \emph{leading coefficient} of $f$ and is denoted
by $\LC(f)$. The \emph{leading monomial} of $f$ is defined as
\( \LM(f) = \LC(f) \, \LT(f) \). If $F \subset \mathcal{R}$ is a finite set
of non-zero polynomials, we denote by
\( \LT(F) = \{\LT(f) \mid f \in F\} \) the set of leading terms of the
elements of $F$. The \emph{leading ideal} of $\mathcal{I}$ is defined as
\( \LT(\mathcal{I}) = \langle \LT(f) \mid 0 \ne f \in \mathcal{I} \rangle
\). A finite set $\{g_1, \dots, g_t\} \subset \mathcal{I}$ of non-zero
polynomials is called a \emph{Gr\"obner basis} of $\mathcal{I}$ with
respect to $\prec$ if
\( \LT(\mathcal{I}) = \langle \LT(g_1), \dots, \LT(g_t) \rangle \). For
more details, we refer to \citep{CLO}.

A natural refinement of Gr\"obner bases is provided by \emph{involutive
  bases}, which are based on the notion of involutive divisions. In this
framework, each element of a finite set of polynomials is assigned a subset
of multiplicative variables and only products of the element with these
multiplicative variables are considered. Since this paper focuses
exclusively on \emph{Pommaret bases} as a special class of involutive
bases, we briefly recall their definition in the sequel. For simplicity, we
will define these bases in the context of monomial ideals. Let
$\mathcal{T}$ represent the set of all terms in $\mathcal{R}$. The class of
each term $x^\alpha \in \mathcal{T}$ with $x^\alpha \ne 1$, denoted
$\cls{(x^\alpha)}$, is the minimal index $i$ such that $\alpha_i\neq 0$. The
set of variables $\{x_1, \ldots, x_{\cls{(x^\alpha)}}\}$ is called the set of
Pommaret multiplicative variables for $x^\alpha$, denoted
${\rm M}_P(x^\alpha)$. Its complement forms the set of Pommaret
non-multiplicative variables, which is denoted ${\rm NM}_P(x^\alpha)$. For
example, we have ${\rm M}_{P}(x_{2}x_{3})=\{x_1,x_2\}$ and
${\rm NM}_{P}(x_{2}x_{3})=\{x_3\}$ (assuming $n=3$).

Let $H \subset \mathcal{T}$ be a finite set of terms. If the monomial ideal
$\mathcal{I} = \langle H \rangle$ satisfies the equality
$\mathcal{I}=\bigoplus_{h\in H}\K[{\rm M}_P(h)]\cdot h$, then $H$ is a {\em
  Pommaret basis} of $\mathcal{I}$.  For a polynomial ideal
$\mathcal{I} \subset \mathcal{R}$, a finite set $H \subset \mathcal{I}$ is
a Pommaret basis, if the leading terms of the elements of $H$ are pairwise
distinct and form a Pommaret basis of $\LT(\mathcal{I})$.  Obviously, any
Pommaret basis is thus also a (generally not reduced) Gr\"obner basis, but
the converse does not hold in general, as Pommaret bases induce in addition
a direct sum decomposition
$\mathcal{I}=\bigoplus_{h\in H}\K\bigl[{\rm M}_P(\LT(h))\bigr]\cdot h$
which is the key to the many special properties of these bases. The main
challenge with Pommaret bases is that they do not always exist. For
example, the principal ideal
$\I=\langle x_{1}x_{2}\rangle\subset\K[x_1,x_2]$ does not possess a
Pommaret basis, as it does not include any elements of class~$2$.

\begin{definition}
  A monomial ideal $\I$ is \emph{quasi-stable} if, for any term
  $x^{\alpha}\in \I$ and for any index $i$ with $\cls{(x^\alpha)}<i\leq n$,
  there exists $s\geq 0$ such that
  $x_i^s x^\alpha/x_{\cls{(x^\alpha)}}\in \I$.
\end{definition}

\begin{proposition}[{\citep[Prop.~5.3.4]{S}}]
  A monomial ideal $\I$ possesses a Pommaret basis, if and only if
  it is quasi-stable.
\end{proposition}

We say that an arbitrary ideal $\I \subset \mathcal{R}$ is in
\emph{quasi-stable position}, if $\lt{(\I)}$ is a quasi-stable ideal and
hence $\I$ has a Pommaret basis.  In \cite{HSS}, it is shown that any ideal
$\I \subset \mathcal{R}$ can be brought into quasi-stable position by a
linear change of coordinates and the paper also presented a deterministic
algorithm to compute such a linear change of coordinates.  Pommaret bases
have many useful applications in algebraic geometry, some of which we
recall here as they are relevant for the remainder of this paper.  One
important application is that these bases allow for a direct reading off of
the Castelnuovo-Mumford regularity of an ideal.

\begin{definition}
Let $\mathcal{I} \subset \mathcal{R}$ be a homogeneous ideal, and let
\[
0 \longrightarrow \bigoplus_{j \in J_r} \mathcal{R}(-e_{r j}) 
\longrightarrow \cdots 
\longrightarrow \bigoplus_{j \in J_1} \mathcal{R}(-e_{1 j}) 
\longrightarrow \bigoplus_{j \in J_0} \mathcal{R}(-e_{0 j}) 
\longrightarrow \mathcal{I} \longrightarrow 0
\]
be its minimal graded free resolution. The \emph{Castelnuovo-Mumford
  regularity} of $\mathcal{I}$ is
$\mathrm{reg}(\mathcal{I}) = \max{ \{ e_{i j} - i \mid i = 0, \dots, r,\ j
  \in J_i \}}$.
\end{definition}

\begin{theorem}[{\cite[Thm.~5.5.15]{S}}]\label{regpom}
  Let $\mathcal{I} \subset \mathcal{R}$ be a homogeneous ideal in
  quasi-stable position. Then, $\mathrm{reg}(\mathcal{I})=\deg{(H)}$ where
  $H$ is the Pommaret basis of $\I$.
\end{theorem}

Pommaret bases allow for an easy determination of the dimension of an
ideal. Let $\mathcal{I} \subset \mathcal{R}$ be an ideal, and let $H$ be a
Pommaret basis of $\mathcal{I}$. According to \citep[Prop.~5.2.3]{S}, the
dimension of $\mathcal{I}$ is given by the minimal integer $i$ such that
for all $j>i$ a pure power of $x_{j}$ appears in $\lt({H})$.

Another important application of Pommaret bases is deciding whether the
quotient ring of a given homogeneous ideal is Cohen-Macaulay.  A finite
sequence $(f_1, \dots, f_k)$ of polynomials in $\mathcal{R}$ is called
\emph{regular} for the $\mathcal{R}$-module $\mathcal{M}$, or briefly
\emph{$\mathcal{M}$-regular}, if $f_1$ is a non-zero divisor on
$\mathcal{M}$, and for each $i$ with $1 < i \le k$, $f_i$ is a non-zero
divisor on the factor module
$\mathcal{M}/\langle f_1, \dots, f_{i-1} \rangle \mathcal{M}$. An
$\mathcal{M}$-regular sequence is called \emph{maximal}, if it cannot be
extended to a longer $\mathcal{M}$-regular sequence. It is well-known that
all maximal regular sequences of $\mathcal{M}$ have the same length, which
is called the \emph{depth} of $\mathcal{M}$ and is denoted by
$\mathrm{depth}(\mathcal{M})$.  A homogeneous ideal in
$\mathcal{I}\subset\mathcal{R}$ is called \emph{Cohen-Macaulay}, if
\( \dim(\mathcal{R}/\mathcal{I}) = \mathrm{depth}(\mathcal{R}/\mathcal{I})
\). Now, suppose that $\mathcal{I}$ is a quasi-stable monomial ideal. Let
$H\subset \mathcal{T}$ be its Pommaret basis, and set
$d = \min_{h \in H} {\cls{(\lt({h}))}}$. Then $(x_1, \dots, x_{d-1})$ forms
a maximal regular sequence for $\mathcal{R}/\mathcal{I}$ as an
$\mathcal{R}$-module, and $\mathrm{depth}(\mathcal{R}/\mathcal{I}) = d-1$;
see \citep[page~179]{S}. Hence, for a quasi-stable monomial ideal
$\mathcal{I}$, the Pommaret basis provides an effective criterion for
deciding whether $\mathcal{R}/\mathcal{I}$ is Cohen-Macaulay.

For later use, we recall some basic facts about generic initial ideals.
Let $\K$ be an infinite field and let $\mathcal I\subset\mathcal R$ be a
homogeneous ideal.  An invertible $n\times n$ matrix $A=(a_{ij})$ over $\K$
acts on $\mathcal R$ by the linear change of coordinates
\[
  A\cdot f
  =f\left(\sum_{i=1}^n a_{i1}x_i,\ldots,
           \sum_{i=1}^n a_{in}x_i\right).
\]
Below, for simplicity, we identify an $n\times n$ matrix with the
corresponding $n^2$-tuple in $\K^{n^2}$. The following classical result
\citep{zbMATH03465689,bs:reverse} shows that the initial ideal of
$A\cdot\mathcal I$ is constant for a generic choice of $A$.

\begin{theorem}
  \label{thm:gin-existence}
  For every homogeneous ideal $\mathcal I\subset \mathcal R$ and every term
  ordering $\prec$, there exists a non-empty Zariski-open subset
  $\mathcal U\subset \K^{n^2}$ such that
  \[
    \lt{(A\cdot\mathcal I)}=\lt{(B\cdot\mathcal I)},
    \qquad\text{for all }A,B\in\mathcal U.
\]
\end{theorem}

The common monomial ideal in this theorem is the \emph{generic initial
  ideal} of $\mathcal I$, denoted by $\operatorname{Gin}(\mathcal I)$.
Here and in what follows, the generic initial ideal is always taken with
respect to the ordering $\prec_{drl}$.  We say that $\mathcal I$ is in
\emph{GIN position} if $\lt(\mathcal I)=\operatorname{Gin}(\mathcal I)$.
The GIN position is usually considered the strongest notion of genericity,
as it implies most other generic positions, but we note that it is shown in
\cite{HSS} that it does not entail componentwise quasi-stable position
(which arises in the context of componentwise linear ideals).  It is
relevant in algebraic geometry, since several important invariants of
$\mathcal I$ coincide with those of $\operatorname{Gin}(\mathcal I)$ (for
most of these invariants actually already quasi-stable position is
sufficient \cite{quasigen}). In particular, the ideals $\mathcal I$ and
$\operatorname{Gin}(\mathcal I)$ have the same Hilbert function and the
same Castelnuovo-Mumford regularity. Moreover, the maximal degrees of a
generator of $\mathrm{Gin}(\I)$ equals $\mathrm{reg}(\I)$; the
Castelnuovo-Mumford regularity of $\I$. This observation provides an
effective control on the computation of $\mathrm{Gin}(\I)$. We refer to
\citep{zbMATH04014856,herzog} for further properties of generic initial
ideals.

We continue with a discussion of the application of Pommaret bases to
computing generic annihilator numbers. We note that these numbers were
originally introduced by Trung~\cite{zbMATH04043980} and subsequently
studied e.\,g.\ by Aramova and Herzog \cite{ah:almreg} and by Conca et
al.~\cite{zbMATH02113823}.  For a book reference on this matter, see
\cite[page 63]{herzog}.

\begin{definition}
  A linear form $\ell\in\mathcal{R}$ is called {\em
    quasi-regular}\footnote{We use here the terminology "quasi-regular"
    following a letter by Serre appended to an article by Guillemin and
    Sternberg \cite{gs:alg} from 1964, which is not widely known among
    algebraists.  The same concept was much later reinvented in algebra by
    various authors under names such as "almost regular" or
    "filter-regular" \cite{ah:almreg,stc:vcm}.} for a graded
  $\mathcal{R}$-module~$\mathcal{M}$, if the graded module
  $ 0:_{\mathcal{M}}\ell = \{ m \in \mathcal{M} \mid \ell \cdot m = 0 \} $
  has finite length (i.\,e.\ only finitely many of its graded components
  are non-zero).  A finite sequence of linear forms
  $(\ell_1,\ldots,\ell_n)$ in $\mathcal{R}$ is \emph{quasi-regular} for the
  module $\mathcal{M}$, if each $\ell_i$ is quasi-regular for the module
  $\mathcal{M}/\langle \ell_1, \ldots, \ell_{i-1} \rangle \mathcal{M}$.
\end{definition}

Quasi-regular sequences are closely related to the existence of a Pommaret
basis.  The following observation in particular implies that if
$\boldsymbol{\ell}$ is a quasi-regular sequence, then using its elements as
new variables $\mathbf{y}$ puts the ideal $\mathcal{I}$ into quasi-stable
position.

\begin{proposition}[{\cite[Thm.~5.2]{wms:spencer2}}]\label{prop4.16}
  Let $\mathcal{I}\subset\mathcal{R}$ be a homogeneous ideal.  Then
  $\mathcal{I}$ possesses a Pommaret basis, if and only if the sequence
  $(x_1,\ldots,x_n)$ is quasi-regular for the module
  $\mathcal{R}/\mathcal{I}$.
\end{proposition}

\begin{definition}
  Given a graded module $\mathcal{M}=\mathcal{R}/\mathcal{I}$ and a
  sequence $\boldsymbol{\ell}=(\ell_1,\ldots,\ell_n)$ of $n$ linearly
  independent, linear forms, we introduce for any index $0\leq i<n$ the
  following $n$ graded modules
  $$
  A_{i}(\boldsymbol{\ell};\mathcal{M}) = 0:_{\mathcal{R}/\langle
    \mathcal{I},\ell_1,\ldots,\ell_{i}\rangle} \ell_{i+1} \;\cong\;
  \bigl(\langle \mathcal{I},\ell_1,\ldots,\ell_{i}\rangle :
  \ell_{i+1}\bigr)/\langle \mathcal{I},\ell_1,\ldots,\ell_{i}\rangle.
  $$
  The {\em annihilator numbers} of $\mathcal{M}$ with respect to the
  sequence $\boldsymbol{\ell}$ are defined by
  $\alpha_{ij}(\boldsymbol{\ell};\mathcal{M}) = \dim_{\K}
  A_{i}(\boldsymbol{\ell};\mathcal{M})_j $ for all indices $j \ge 0$ and
  $0 \le i < n$.  Here, for a graded $\mathcal{R}$-module $X$,
  $\dim_{\K}(X_j)$ denotes the dimension of its degree $j$ component $X_j$
  considered as a $\K$-linear space.
\end{definition}

Like Betti numbers, the annihilator numbers
$\alpha_{ij}(\boldsymbol{\ell};\mathcal{M}) $ can be organised into a table
with the first index enumerating the columns and the second index the rows.
The following new proposition characterises when the table has only
finitely many rows.  Its proof requires some material on Janet bases and
their relation to Pommaret bases which can all be found in
\cite[Chapts.~3\&4]{S}.

\begin{proposition}\label{fintab}
  Only finitely many annihilator numbers
  $\alpha_{ij}(\boldsymbol{\ell};\mathcal{M}) $ are non-zero, if and only
  if the sequence $\boldsymbol{\ell}$ is quasi-regular for $\mathcal{M}$.
\end{proposition}

\begin{proof}
  One direction is obvious, as the definition of quasi-regularity
  immediately implies that only finitely many annihilator numbers can be
  non-zero.

  For the converse, we argue by contradiction.  We use the sequence
  $\boldsymbol{\ell}$ of linear forms as new variables and assume that
  $\boldsymbol{\ell}$ is not quasi-regular.  Then $\mathcal{I}$ is in the
  new variables not in quasi-stable position.  Let $H$ be a Pommaret
  autoreduced Janet basis of~$\mathcal{I}$.  Such a basis always exists,
  but as $\mathcal{I}$ is not in quasi-stable position, it is not
  simultaneously a Pommaret basis.  Because of the Pommaret autoreduction,
  every variable which is Pommaret multiplicative for a generator is also
  Janet multiplicative.  Because of the missing quasi-stability, there
  exists at least one generator $h\in H$ with $\cls{(\LT{(h)})}=k$ for
  which in addition a variable $x_{\ell}$ with $\ell>k$ is Janet
  multiplicative.

  By simply dropping all terms in $h$ of a class less than $k$, we obtain a
  polynomial
  $\bar{h} \in \bigl(h+\langle x_{1},\dots,x_{k-1}\rangle\bigr) \cap
  \langle x_{k}\rangle$.  By construction, the polynomial $g=\bar{h}/x_{k}$
  is an element of the colon ideal
  $\langle\mathcal{I},x_{1},\dots,x_{k-1}\rangle \colon x_{k}$.  We now
  claim that
  $x_{\ell}^{m}g \notin \langle\mathcal{I},x_{1},\dots,x_{k-1}\rangle$ for
  all $m\geq0$ which entails
  $\alpha_{k-1,j}(\boldsymbol{\ell};\mathcal{M})\neq0$ for infinitely many
  $j$ and thus an infinite annihilator table.  For an indirect proof, we
  assume that
  $x_{\ell}^{m}g \in \langle\mathcal{I},x_{1},\dots,x_{k-1}\rangle$ and
  note that by construction $\cls{\bigl(\LT{(x_{\ell}^{m}g)}\bigr)}\geq k$
  and by our assumption $\LT{(x_{\ell}^{m}g)} \in
  \LT{(\mathcal{I})}$. Hence the Janet basis $H$ must contain a unique
  element $h_{m}$ such that $\LT{(h_{m})}$ is a Janet divisor of
  $\LT{(x_{\ell}^{m}g)}$.  Obviously, this entails that
  $\cls{\bigl(\LT{(h_{m})}\bigr)}\geq k$ and hence that $x_{k}$ is Janet
  multiplicative for it so that $\LT{(h_{m})}$ is a Janet divisor of
  $\LT{(x_{\ell}^{m}h)}=x_{k}\LT{(x_{\ell}^{m}g)}$, too.  But since by
  assumption $x_{\ell}$ is Janet multiplicative for $h$, $\LT{(h)}$ is also
  a Janet divisor of $\LT{(x_{\ell}^{m}h)}$ which contradicts the
  uniqueness of involutive divisors, as trivially $h_{m}\neq h$.  Thus our
  claim is true.
\end{proof}

If one uses in the quasi-stable case the given variables as quasi-regular
sequence, then the corresponding annihilator numbers are closely related to
the Pommaret basis $H$ of the ideal $\mathcal{I}$: they simply encode the
degrees and classes of the generators in $H$.

\begin{theorem}[{\cite[Thm. 4.17]{HSS}}]\label{alpha-pom.}
  Let $\mathcal{M}=\mathcal{R}/\mathcal{I}$ for a homogeneous ideal
  $\mathcal{I}\subset\mathcal{R}$ possessing a Pommaret basis
  $H$. Then, for all admissible indices $i,j$, we have
  $$
  \alpha_{ij}(x_1,\ldots,x_n;\mathcal{M}) = \#\bigl\{ h \in H \mid
  \cls{(h)} = i+1,\ \deg{(h)} = j+1 \bigr\}.
  $$
\end{theorem}

\begin{example}
  Let $\mathcal{R}=\mathbb{Q}[x_1,x_2]$,
  $\mathcal{I}=\langle x_2^2,x_1x_2\rangle$, and
  $\mathcal{M}=\mathcal{R}/\mathcal{I}$. It is easy to see that
  $\mathcal{I}$ has the Pommaret basis $H=\{x_2^2,x_1x_2\}$. For the
  sequence $\boldsymbol\ell=(x_1,x_2)$ we have
  $A_0 = 0:_{\mathcal{R}/\mathcal{I}} x_1 \cong
  (\mathcal{I}:x_1)/\mathcal{I} = \langle x_2\rangle/\langle
  x_2^2,x_1x_2\rangle$, which is one-dimensional in degree $1$, and
  $A_1 = 0:_{\mathcal{R}/\langle \mathcal{I},x_1\rangle} x_2 \cong (\langle
  \mathcal{I},x_1\rangle:x_2)/\langle \mathcal{I},x_1\rangle = \langle
  x_1,x_2\rangle/\langle x_1,x_2^2\rangle$, which is also one-dimensional
  in degree $1$.  As all degree 0 components vanish, the only non-zero
  annihilator numbers are $\alpha_{01}(x_{1},x_{2};\mathcal{M}) = 1$ and
  $\alpha_{11}(x_{1},x_{2};\mathcal{M}) = 1$.
\end{example}

One simple consequence of the above theorem is that the annihilator numbers
of $\mathcal R/\mathcal I$ and $\mathcal R/\lt{(\mathcal I)}$ coincide in
quasi-stable position, as the leading terms of the Pommaret basis $H$ form
a Pommaret basis of $\lt{(\mathcal I)}$.  Another one is a simple
(non)vanishing result on the annihilator numbers which follows from the
fact that the minimal class of a generator in a Pommaret basis is the depth
of the ideal $\mathcal{I}$ \cite[Prop.~5.2.7]{S} and thus
$\mathrm{depth}(\mathcal{R/I})+1$.

\begin{corollary}\label{alpgaIlt}
  Let $\mathcal I\subset\mathcal R$ be a homogeneous ideal in quasi-stable
  position.
  \begin{enumerate}[(i)]
  \item We have
    \( \alpha_{ij}\bigl(x_1,\dots,x_n; \mathcal R/\mathcal I\bigr) =
    \alpha_{ij}\bigl(x_1,\dots,x_n; \mathcal R/\lt{(\mathcal I)}\bigr)\)
    for all admissible values of the indices $i,j$.
  \item $\sum_{j\geq 0}\alpha_{ij}(x_1,\ldots,x_n;\mathcal{R/I})=0$, if and
     only if $i<\mathrm{depth}(\mathcal{R/I})$.
  \end{enumerate}
\end{corollary}

These observations entail now the existence of \emph{generic annihilator
  numbers} which are also the annihilator numbers of the generic initial
ideal (thus they provide an alternative proof of \cite[Thm.~4.3.6]{herzog}).

\begin{theorem}\label{genann}
  Any graded module $\mathcal{M}=\mathcal{R}/\mathcal{I}$ possesses only
  finitely many different finite tables of annihilator numbers and almost
  all choices of the sequence $\boldsymbol{\ell}=(\ell_1,\ldots,\ell_n)$
  lead to the same table defined by the generic initial ideal:
  \begin{displaymath}
    \alpha_{ij}(\ell_1,\ldots,\ell_n;\mathcal{R/I}) =
    \alpha_{ij}\bigl(x_1,\ldots,x_n;\mathcal{R}/\mathrm{Gin}(\mathcal{I})\bigr).
  \end{displaymath}
\end{theorem}

\begin{proof}
  By Proposition~\ref{fintab}, finite annihilator tables occur, if and only
  if the sequence~$\boldsymbol{\ell}$ is quasi-regular. In this case, we
  can use the linear forms $\ell_{i}$ as new variables.  By
  Proposition~\ref{prop4.16}, the ideal $\mathcal{I}$ will possess a
  Pommaret basis in these variables and, by Theorem~\ref{alpha-pom.}, the
  annihilator numbers $\alpha_{ij}(\ell_1,\ldots,\ell_n;\mathcal{M})$
  encode the degrees and classes of the generators in this basis.  It is a
  simple consequence of the theory of Gr\"obner systems \cite{vw:compgb} that
  any polynomial ideal possesses only finitely many different leading
  ideals under linear transformations of the variables.  Hence, by
  Corollary~\ref{alpgaIlt}, the module $\mathcal{R}/\mathcal{I}$ has only
  finitely many different finite annihilator tables.  Obviously, the
  transformation defined by any generic choice of the
  sequence~$\boldsymbol{\ell}$ will lead to the generic initial ideal
  $\mathrm{Gin}(\mathcal{I})$ and thus always to the same table.
\end{proof}

Finally, we present in this section a novel proof of the following
well-known result on the behaviour of the regularity under generic
hyperplane sections (see e.\,g.\ {\cite[Prop.~20.20]{zbMATH00704831}}).  It
relates directly the involved invariants to the Pommaret basis, an
observation of independent interest.

\begin{proposition}\label{regl}
  Let $\ell\in\mathcal{R}$ be a quasi-regular linear form for the finitely
  generated, graded $\mathcal{R}$-module $\mathcal{M}$.  Then
  \begin{displaymath}
    \mathrm{reg}(\mathcal{M}) = 
    \max{\bigl\{\mathrm{reg}(\mathcal{M}/\ell\mathcal{M}),
      \mathrm{reg}(0:_{\mathcal{M}}\ell)\bigr\}}.
  \end{displaymath}
\end{proposition}

For notational simplicity, we will prove this proposition only for the case
$\mathcal{M}=\mathcal{R}/\mathcal{I}$ with a homogeneous ideal
$\mathcal{I}\subset\mathcal{R}$; the extension to general, finitely
generated modules $\mathcal{M}=\mathcal{R}^m/\mathcal{U}$ is
straightforward.  Since $\ell$ is assumed to be quasi-regular, we can
always perform a linear change of variables such that $\mathcal{I}$ is in
quasi-stable position and $\ell$ is the first variable.  Hence we may
assume without loss of generality that $\mathcal{I}$ is in quasi-stable
position and $\ell=x_1$.  Let $H$ be the Pommaret basis of $\mathcal{I}$.
We disjointly decompose $H=H_1\cup H_{>1}$ where $H_1$ contains all generators
$h\in H$ of class~$1$ and $H_{>1}$ the remaining generators of higher class.
Proposition~\ref{regl} follows now immediately from Theorem~\ref{regpom}
and the following result.

\begin{proposition}\label{reg12}
    With the above notations, we have the following equalities:
    \begin{enumerate}[(i)]
        \item $\mathrm{reg}(\mathcal{M}/x_1\mathcal{M})=\deg{(H_{>1})}-1$.
        \item \label{reg12-2} $\mathrm{reg}(0:_{\mathcal{M}}x_1)=\deg{(H_1)}-1$.
    \end{enumerate}
\end{proposition}

\begin{proof}
  (i) follows trivially from the observations that
  $\mathcal{M}/x_1\mathcal{M} \cong
  \mathcal{R}/\langle\mathcal{I},x_1\rangle \cong
  \mathcal{K}[x_2,\dots,x_n]/\mathcal{I}|_{x_1=0}$ and that $H_{>1}|_{x_1=0}$
  is the Pommaret basis of $\mathcal{I}|_{x_1=0}$
  \cite[Lemma~3.1]{zbMATH06126804} together with Theorem~\ref{regpom}.

  For proving (ii), we need the theory of relative involutive bases
  developed in \cite{hos:relgb} and the obvious isomorphism
  $0:_{\mathcal{M}}x_1\cong (\mathcal{I}\colon x_1)/\mathcal{I}$.  Since we
  work with the degree reverse lexicographic ordering, every polynomial of
  class~$1$ is contained in $\langle x_1\rangle$ and hence we can introduce
  the set $H'_1=\{h/x_1\mid h\in H_1\}$.  We first claim that
  $H'=H'_1\cup H_{>1}$ is the Pommaret basis of $\mathcal{I}\colon x_1$.
  Given a polynomial $f\in\mathcal{I}\colon x_1$, the polynomial $x_1f$
  lies in $\mathcal{I}$ and hence possesses an involutive standard
  representation $x_1f=\sum_{h\in H}q_hh$ with
  $q_h\in\mathcal{K}\bigl[{\rm M}_P(\LT(h))\bigr]$.  Because of the factor
  $x_1$, we must have $q_h\in\langle x_1\rangle$ for all $h$ of class
  greater than~$1$.  This leads to an involutive standard representation of
  $f$ over $H'$, $f=\sum_{h\in H_1}q_h(h/x_1)+\sum_{h\in H_{>1}}(q_h/x_1)h$,
  entailing our claim.

  A simple consequence of the just proven claim is that $H'_1$ is a
  Pommaret basis of $\mathcal{I}\colon x_1$ relative to $\mathcal{I}$ (and
  thus can be considered as a Pommaret basis of the factor module
  $(\mathcal{I}\colon x_1)/\mathcal{I}$). \cite[Thm.~5.13]{hos:relgb}
  provides now a relative Pommaret basis of the first syzygy module of
  $(\mathcal{I}\colon x_1)/\mathcal{I}$ considered as an
  $\mathcal{R}/\mathcal{I}$-module.  It consists of two types of elements:
  $S$- and $A$-syzygies.  For the $S$-syzygies, we start with the
  involutive standard representations $x_kh_1=\sum_{h\in H'}q_hh$ with
  $q_h\in\mathcal{K}\bigl[{\rm M}_P(\LT(h))\bigr]$ where $h_1\in H'_1$ and
  $k>\cls{(h_1)}$ (which exist since we  showed above that $H'$ is a
  Pommaret basis of $\mathcal{I}\colon x_1$). They show that
  $S_{h_1,k}=x_kh_1-\sum_{h\in H'_1}q_h(0)h\in\mathcal{I}$. The
  $A$-syzygies are the trivial relations $A_{h_1}=x_1h_1\in\mathcal{I}$ for
  all $h_1\in H'_1$.  The mentioned relative Pommaret basis of
  $\mathrm{Syz}_{\mathcal{R}/\mathcal{I}}\bigl((\mathcal{I}\colon
  x_1)/\mathcal{I}\bigr)$ is then obtained by taking the residue classes of
  all coefficients of these syzygies modulo $\mathcal{I}$.  We are treating
  $(\mathcal{I}\colon x_1)/\mathcal{I}$ as an $\mathcal{R}$-module and it
  is easy to see that the set of all the linear combinations $S_{h_1,k}$,
  $A_{h_1}$ induces an ordinary (i.\,e.\ not relative) Pommaret basis
  $H_{\mathrm{Syz}}$ of
  $\mathrm{Syz}_{\mathcal{R}}\bigl((\mathcal{I}\colon
  x_1)/\mathcal{I}\bigr)$.

  The Pommaret basis $H_{\mathrm{Syz}}$ of
  $\mathrm{Syz}_{\mathcal{R}}\bigl((\mathcal{I}\colon
  x_1)/\mathcal{I}\bigr)$ in turn induces a whole free resolution of the
  syzygy module so that in total we have constructed a (generally not
  minimal) free resolution of $(\mathcal{I}\colon x_1)/\mathcal{I}$.  We
  note that by construction $\deg{(H_{\mathrm{Syz}})}=\deg{(H_1)}$. Hence,
  by Theorem~\ref{regpom}, the regularity of
  $\mathrm{Syz}_{\mathcal{R}}\bigl((\mathcal{I}\colon
  x_1)/\mathcal{I}\bigr)$ equals $\deg{(H_1)}$ and consequently
  $\mathrm{reg}(0:_{\mathcal{M}}x_1)=\deg{(H_1)}-1$ as claimed.
\end{proof}

\begin{example}
  As a simple example for the construction used in the above proof, we
  consider the monomial ideal
  $\mathcal{I} = \langle x_1^2x_2, x_2^3, x_2x_3, x_3^2\rangle \subset
  \mathbb{Q}[x_1,x_2,x_3]$ and the module
  $\mathcal{M}=\mathcal{R}/\mathcal{I}$.  For the Pommaret basis of
  $\mathcal{I}$, we only have to add the element $x_1^2x_2^2$, so that here
  $H_1=\{x_1^2x_2, x_1^2x_2^2\}$ and $H_{>1}=\{ x_2^3, x_2x_3, x_3^2\}$.
  Now
  $\mathcal{M}/x_1\mathcal{M} \cong \mathcal{K}[x_2,x_3] / \langle
  H_{>1}\rangle$ and thus $\mathrm{reg}(\mathcal{M}/x_1\mathcal{M})=2$ by
  Theorem~\ref{regpom}.

  The module
  $0:_{\mathcal{M}}x_1 \cong (\mathcal{I}\colon
  x_1)/\mathcal{I}$ is here a cyclic module generated by the residue class
  $x_1x_2+\mathcal{I}$; for a $\mathcal{K}$-linear basis one has to add the
  residue class $x_1x_2^2+\mathcal{I}$ induced by the additional generator
  of the Pommaret basis.  Its first syzygy module is generated by four
  $S$-syzygies corresponding to the non-multiplicative prolongations
  $x_2\cdot x_1^2x_2$, $x_2\cdot x_1^2x_2^2$ and $x_3\cdot x_1^2x_2$,
  $x_3\cdot x_1^2x_2^2$ and by two $A$-syzygies corresponding to
  $x_1\cdot x_1^2x_2$, $x_1\cdot x_1^2x_2^2$.  In total, this yields the
  free resolution
  \begin{displaymath}
    0 \longrightarrow \mathcal{R}^2 \stackrel{d_3}{\longrightarrow} 
    \mathcal{R}^6 \stackrel{d_2}{\longrightarrow}
    \mathcal{R}^6 \stackrel{d_1}{\longrightarrow}
    \mathcal{R}^2 \longrightarrow (\mathcal{I}\colon x_1)/\mathcal{I} 
    \longrightarrow 0
  \end{displaymath}
  with the differentials
  \begin{gather*}
    d_1 =
    \begin{pmatrix}
      x_1 & 0 & x_2 & 0 & x_3 & 0 \\
      0 & x_1 & -1 & x_2 & 0 & x_3
    \end{pmatrix},\\[0.5\baselineskip]
    d_2 =
    \begin{pmatrix}
      x_2 & 0 & x_3 & 0 & 0 & 0 \\
      -1 & x_2 & 0 & x_3 & 0 & 0 \\
      -x_1 & 0 & 0 & 0 & x_3 & 0 \\
      0 & -x_1 & 0 & 0 & 0 & x_3 \\
      0 & 0 & -x_1 & 0 & -x_2 & 0 \\
      0 & 0 & 0 & -x_1 & 1 & -x_2
    \end{pmatrix},\qquad
    d_3 = 
    \begin{pmatrix}
      x_3 & 0 \\
      0 & x_3 \\
      -x_2 & 0 \\
      1 & -x_2 \\
      x_1 & 0 \\
      0 & x_1
    \end{pmatrix}.
  \end{gather*}
  The minimal resolution
  \begin{displaymath}
    0 \longrightarrow \mathcal{R} \longrightarrow \mathcal{R}^3
    \longrightarrow \mathcal{R}^3 \longrightarrow \mathcal{R} 
    \longrightarrow (\mathcal{I}\colon x_1)/\mathcal{I} \longrightarrow 0
  \end{displaymath}
  can either be obtained by minimisation or by observing that it is given by
  the Koszul complex over the regular sequence $x_1,x_2^2,x_3$ (one can
  easily check that
  $(\mathcal{I}\colon x_1)/\mathcal{I} \cong \mathcal{R}/\langle
  x_1,x_2^2,x_3\rangle$, as
  $\mathcal{I}\colon x_1x_2=\langle x_1,x_2^2,x_3\rangle$). It follows that
  $\mathrm{reg}(0:_{\mathcal{M}}x_1)=3$ and hence also
  $\mathrm{reg}(\mathcal{M})=3$, as predicted by the Pommaret basis.
\end{example}

As a consequence of the above considerations, one obtains immediately the
following equality which seems to be new. We denote by $\operatorname{sat}(\mathcal I)$ the satiety of
$\mathcal I$, i.e. the smallest degree from which $\mathcal I$ and its
saturation (with respect to the maximal homogeneous ideal) agree.

\begin{corollary}
  Let $\ell\in\mathcal{R}$ be a quasi-regular linear form for the module
  $\mathcal{M}=\mathcal{R}/\mathcal{I}$ with a homogeneous ideal
  $\mathcal{I}\subset\mathcal{R}$.  Then
    \begin{displaymath}
        \mathrm{reg}(0:_{\mathcal{M}}\ell) = \mathrm{sat}(\mathcal{I})-1.
    \end{displaymath}
\end{corollary}

\begin{proof}
  In the notations used above, we have
  $\mathrm{sat}(\mathcal{I})=\deg{(H_1)}$ \cite[Cor.~5.5.29]{S}.  Now the
  claim follows immediately from Proposition~\ref{reg12} \ref{reg12-2}.
\end{proof}

%\begin{proposition}\label{4.34-ths}
   % Let $\mathcal{I}\subset\mathcal{R}$ be a monomial ideal. Then, $\mathcal{I}$ is stable if and only if its minimal basis coincides with its Pommaret basis.
%\end{proposition}
%\begin{proposition}\label{4.38-ths}
  %  Let $\mathcal{I}\subset \mathcal{R}$ be a homogeneous ideal generated in degrees at most $q$. Then, $\mathcal{I}$ is $q$-regular if and only if $\mathcal{I}_{\geq q}$ is stable.
  %\end{proposition}

%%% Local Variables:
%%% mode: latex
%%% TeX-master: "MainPaperFile"
%%% End:

%% file: StronglyStable.tex
\section{Strongly stable ideals}
\label{sec:3}

This section reviews auxiliary results on strongly stable ideals that are
relevant to the subject of this paper. Although most of them were already
proven in \citep{axial}, we provide new proofs for them using the theory of
Pommaret bases.

\begin{definition}
  The monomial ideal $\mathcal{I}\subset \mathcal{R}$ is \emph{strongly
    stable} if for every term $x^\alpha\in\mathcal{I}$, and every index
  pair $i<j$ such that $x_i\mid x^\alpha$, the term $x^\alpha x_j/x_i$ also
  lies in $\mathcal{I}$.  For an arbitrary polynomial ideal $\mathcal{I}$,
  we say that it is in \emph{strongly stable position}, if its leading
  ideal $\lt(\mathcal{I})$ is strongly stable.
\end{definition}
  
We begin by establishing some useful relations among the exponents of the
elements in the Pommaret basis of a strongly stable ideal.  Throughout this
section, unless explicitly stated otherwise, let
$\mathcal{I}\subset \mathcal{R}$ be a strongly stable monomial ideal of
dimension~$D$ and let $H$ denote its Pommaret basis. Set $c = D+1$ and let
$q$ denote the maximal degree of the elements of $H$.  Since every strongly
stable ideal is, in particular, quasi-stable, it follows that
\( \operatorname{reg}(\mathcal{I}) = q \).

\begin{lemma}\label{ss-1}
  Let $\mathcal{I}\subset\mathcal{R}$ be a quasi-stable ideal, and let $H$
  denote its Pommaret basis. Then, the smallest index $i$ such that there
  exists $a_{i},\ldots,a_n\in \mathbb{N}$ with
  $x_{i}^{a_{i}}, \ldots, x_n^{a_n}\in H$ is $c$. In addition, when
  $\mathcal{I}$ is strongly stable, the exponents satisfy
  $a_n \leq \cdots \leq a_{c}$.
\end{lemma}

\begin{proof}
  The first assertion is a well-known fact about quasi-stable ideals.
  Since $x_c^{a_c}\in H$ and $\mathcal{I}$ is strongly stable, we have
  $x_i^{a_c}\in \mathcal{I}$ for every $i>c$. Because $H$ is a Pommaret
  basis of $\mathcal{I}$, the term $x_i^{a_c}$ is divisible by some element
  of $H$; necessarily, this divisor is $x_i^{a_i}$, so
  $x_i^{a_i}\mid x_i^{a_c}$, and hence $a_i\le a_c$. Now fix indices
  $c\le j<i\le n$. Since $x_j^{a_j}\in H$ and $\mathcal{I}$ is strongly
  stable, we also have $x_i^{a_j}\in \mathcal{I}$. By the same argument,
  $x_i^{a_i}\mid x_i^{a_j}$, and therefore $a_i\le a_j$. This proves the
  second claim.
\end{proof}
  
In the following, the symbols $a_c,\ldots,a_n$ always refer to the
exponents introduced in Lemma~\ref{ss-1}.  It is straightforward to see
that no element of a quasi-stable ideal $\mathcal{I}$ can depend only on
the variables $x_1,\ldots,x_D$.
  
\begin{lemma}\label{ss-2}
  For each mixed term $x_k^{b_k}\cdots x_n^{b_n}\in H$ with $c\leq k< n$
  and for every $k<i\leq n$, the inequality $\sum_{j=i}^nb_j<a_i$ holds.
\end{lemma}

\begin{proof}
  We argue by induction over $i$.  For $i=n$, we must have $b_n<a_n$.
  Indeed, if we had $a_n\le b_n$, then
  \( x_n^{a_n} \mid_P x_k^{b_k}\cdots x_n^{b_n}\) contradicting the
  assumption $x_k^{b_k}\cdots x_n^{b_n}\in H$. Now, let $k<i<n$ and assume
  that the assertion holds for $i+1,\ldots,n$. Suppose, for contradiction,
  that \( a_{i}\le\sum_{j=i}^n b_j \).  By the inductive hypothesis, the
  made assumption, and the preceding lemma, we obtain
  \[
    b_{i+1}+\cdots+b_n < a_{i+1}\leq a_{i}\leq b_{i}+b_{i+1}+\cdots+b_n.
  \]
  Since $\mathcal{I}$ is strongly stable, $x_{i}^{a_{i}}\in \mathcal{I}$
  implies that
  \( x_{i}^{a_{i}-(b_{i+1}+\cdots+b_n)}x_{i+1}^{b_{i+1}}\cdots x_n^{b_n}\in
  \mathcal{I} \). Consequently,
  \( x_{i}^{a_{i}-(b_{i+1}+\cdots+b_n)}x_{i+1}^{b_{i+1}}\cdots x_n^{b_n}
  \mid_P x_{k}^{b_{k}}x_{k+1}^{b_{k+1}}\cdots x_n^{b_n}, \) which
  contradicts the assumptions that
  $x_{k}^{b_{k}}x_{k+1}^{b_{k+1}}\cdots x_n^{b_n}\in H$ and that $H$ is a
  Pommaret basis of $\mathcal{I}$.
\end{proof}

The following lemma extends this inequality to the case $i = k$.

\begin{lemma}\label{ss-3}
  For every mixed term $x_k^{b_k}\cdots x_n^{b_n}\in H$ with $c\leq k< n$,
  we have the inequalities
  $a_n\leq \sum_{i=k}^nb_i\leq a_k$.
\end{lemma}

\begin{proof}
  For an indirect proof of the first inequality, we assume
  $\sum_{i=k}^n b_i < a_n$.  Since $\mathcal{I}$ is strongly stable and
  $x_k^{b_k}\cdots x_n^{b_n}\in H$, it follows that
  $x_n^{\sum_{i=k}^n b_i}\in \mathcal{I}$.  This contradicts the facts that
  $x_n^{a_n}\in H$ and that $H$ is a Pommaret basis of
  $\mathcal{I}$. Hence, $a_n \leq \sum_{i=k}^n b_i$.

  To prove the other inequality, we note that, by Lemmata~\ref{ss-1} and
  \ref{ss-2}, we have \( \sum_{i=k+1}^n b_i < a_{k+1} \leq a_k, \) implying
  that \( 0 < a_k - \sum_{i=k+1}^n b_i.  \) Since $\mathcal{I}$ is strongly
  stable and $x_k^{a_k}\in H$, it follows that
  \( x_k^{\,a_k-\sum_{i=k+1}^n b_i} x_{k+1}^{b_{k+1}}\cdots x_n^{b_n} \in
  \mathcal{I} \). Because $x_k^{b_k}\cdots x_n^{b_n}\in H$, we must have
  \( b_k \leq a_k - \sum_{i=k+1}^n b_i \), which shows the desired
  inequality.
\end{proof}

\begin{lemma}\label{ss-4}
    $\mathrm{reg}\bigl(\mathcal{I}+\langle x_1,\ldots,x_{D}\rangle\bigr)=a_c$.
  \end{lemma}
  
\begin{proof}
  By adding the variables $x_1, \ldots, x_{D}$ to the ideal, only terms of
  class at least $c$ remain unchanged. Accordingly, the ideal
  $\mathcal{I} + \langle x_1, \ldots, x_{D} \rangle$ is generated by the
  variables $x_1, \ldots, x_{D}$, the pure powers
  $x_c^{a_c}, \ldots, x_n^{a_n}$, and the mixed terms of the form
  $x_k^{b_k}\cdots x_n^{b_n}$ with $c \leq k < n$.  Since
  $\mathcal{I}+\langle x_1,\ldots,x_{D}\rangle$ is zero-dimensional, it
  follows that this ideal is quasi-stable. Now, we proceed to compute the
  degree of its Pommaret basis. We claim that the elements obtained by
  multiplying the generators $x_1,\ldots,x_D$ with some non-multiplicative
  variables $x_c, \ldots,x_n$ have degree at most $a_c$.

  Suppose, for the sake of contradiction, that a term
  $x_c^{b_c}\cdots x_n^{b_n}$ of degree at least $a_c$ is multiplied by
  some variable $x_j$ with $1 \leq j \leq D$, and that
  $x_j \cdot x_c^{b_c}\cdots x_n^{b_n}$ has no Pommaret divisor in
  $H$. Thus, $a_c\le \sum_{i=c}^n b_i$. Note that all variables
  $x_c, \ldots, x_n$ are non-multiplicative for $x_j$. Moreover, by
  Lemma~\ref{ss-1}, we have $b_n \leq a_n \leq a_c$.

  Since $\mathcal{I}$ is strongly stable and $x_c^{a_c}\in \mathcal{I}$,
  there exists a term
  $x^\gamma := x_k^{c_k} x_{k+1}^{b_{k+1}} \cdots x_n^{b_n}$ with
  $k \geq c$, satisfying $c_k + \sum_{i=k+1}^n b_i = a_c$ and
  $c_k \leq b_k$, such that $x^\gamma \in \mathcal{I}$.  Since
  $x_c, \ldots, x_k$ are multiplicative for $x^\gamma$, it follows that
  $x_c^{b_c}\cdots x_n^{b_n}\in \mathcal{I}$. Furthermore, because $x_j$ is
  multiplicative for this term, we conclude that
  $x_j x_c^{b_c}\cdots x_n^{b_n}$ has a Pommaret divisor in $H$,
  contradicting our assumption.

  Hence, the claim follows: the terms that must be multiplied with $x_j$
  for $1 \leq j \leq D$ in the computation of the Pommaret basis have a
  degree at most $a_c - 1$. Consequently, by Lemma \ref{ss-1} the Pommaret
  basis of $\mathcal{I} + \langle x_1, \ldots, x_{D} \rangle$ has degree
  $a_c$, and the proof is complete.
\end{proof}

\begin{proposition}\label{ss-5}
    If $\mathcal{R}/\mathcal{I}$ is Cohen-Macaulay, then $\mathrm{reg}(\mathcal{I})=a_c$.
  \end{proposition}
  
\begin{proof}
  We claim that $H$ contains no mixed term of the form
  $x_k^{b_k}\cdots x_n^{b_n}$ with $k \leq D$.  Suppose, for contradiction,
  that $H$ contains such a term. Since $\mathcal{R}/\mathcal{I}$ is
  Cohen--Macaulay, the properties of Pommaret bases imply that
  $\langle H, x_1, \ldots, x_{k-1} \rangle_{q} = \mathcal{R}_{q}$.  On the
  other hand, by definition,
  $D = \min{\{i \mid \langle H, x_1, \ldots, x_i \rangle_{q} =
    \mathcal{R}_{q}\}}$ and since $k-1 < D = \dim(\mathcal{I})$, we obtain
  a contradiction. Therefore, the claim holds.  Thus, no mixed term in $H$
  has class less than $c$, and we conclude that
  $\min_{h \in H} \cls(h) = c$.  By Lemma~\ref{ss-3}, if
  $x_k^{b_k}\cdots x_n^{b_n} \in H$ with $c \leq k < n$, then
  $\sum_{i=k}^{n} b_i \leq a_k$. Combining this result with
  Lemma~\ref{ss-1}, we deduce that the maximal degree of an element of $H$
  is $a_c$ and $\mathrm{reg}(\mathcal{I}) = a_c$.
\end{proof}

If $\mathcal{R}/\mathcal{I}$ is not Cohen-Macaulay, then $H$ contains
elements of class less than \(c\), which are necessarily mixed terms. But
if we assume that none of them has a degree larger than $a_c$, then the
claim of the proposition remains true.

\begin{corollary}\label{ss-6}
  If any term in $H$ involving at least one of the variables
  $x_1, \ldots, x_{D}$ has degree at most $a_c$, then
  $\mathrm{reg}(\mathcal{I}) = a_c$.
\end{corollary}

We now proceed to recall the notion of partial degree for an ideal and
study its relation with the Castelnuovo-Mumford regularity, see
\citep{axial} for more details.

\begin{definition}\label{pd-def}
  Let $\mathcal{I}\subset \mathcal{R}$ be a monomial ideal with minimal
  generating set $G$.  For each $1\leq i\leq n$, we define
  $\omega_i(\mathcal{I})=\max {\{\mathrm{d}_i(x^\alpha)\mid x^\alpha \in
    G\}}$ where the $i$-th \emph{partial degree} of $x^\alpha$ with
  $\alpha=(\alpha_1,\ldots ,\alpha_n)$ is given by
  $\mathrm{d}_i(x^\alpha)=\sum_{j\geq i}\alpha_j$.
\end{definition}

\begin{proposition}\label{omega-1}
  Let $\mathcal{I}$ be a strongly stable ideal. For all indices
  $1\leq i\leq n$, we have
  $$\mathrm{reg}\,\bigl(\mathcal{I}+\langle
  x_1,\ldots,x_{i-1}\rangle\bigr)=\omega_i(\mathcal{I})\,.$$
\end{proposition}

\begin{proof}
  We first note that the Pommaret basis and the minimal generating set of a
  strongly stable ideal coincide.  We then distinguish two cases based on
  the value of $i$.

  \textbf{(Case 1)} Assume that $1\leq i\leq D$.  In order to prove the
  desired equality, we transform
  $\mathcal{I}+\langle x_1,\ldots,x_{i-1}\rangle$ into a quasi-stable ideal
  by a suitable permutation of variables.  We define an invertible linear
  change of variables $\varphi$ as follows: For any $c \leq j \leq n$, set
  $x_j \mapsto y_j$, and for the remaining variables we define
  \begin{align*}
    x_1 &\mapsto y_{c-i+1}, \quad
          x_2 \mapsto y_{c-i+2}, \quad \ldots, \quad
          x_{i-1} \mapsto y_{c-1}, \\
    x_i &\mapsto y_1, \quad
          x_{i+1} \mapsto y_2, \quad \ldots, \quad
          x_{c-1} \mapsto y_{c-i}.
  \end{align*}
  As a result, the image of
  \(\langle H \rangle + \langle x_1, \ldots, x_{i-1} \rangle\) under this
  change becomes
  $\langle \varphi(H) \rangle + \langle y_{c-i+1}, \ldots, y_{c-1}
  \rangle$.  Since an invertible linear change of variables preserves the
  regularity of ideals, we obtain
  $\mathrm{reg}(\mathcal{I} + \langle x_1, \ldots, x_{i-1} \rangle) =
  \mathrm{reg}(\varphi(\mathcal{I}) + \langle y_{c-i+1}, \ldots, y_{c-1}
  \rangle)$. We claim that the Pommaret basis of
  $\varphi(\mathcal{I}) + \langle y_{c-i+1}, \ldots, y_{c-1} \rangle$ is
  finite, and hence the resulting ideal is quasi-stable. Then, by computing
  the degree of its Pommaret basis, we can determine the regularity of the
  original ideal.

  We show that the Pommaret basis of
  $\varphi(\mathcal{I}) + \langle y_{c-i+1}, \ldots, y_{c-1} \rangle$ can
  be obtained from $\varphi(H)$, the variables $y_{c-i+1},\ldots,y_{c-1}$,
  and the elements obtained by multiplying these variables by their
  non-multiplicative variables. Note that the non-multiplicative variables
  for each $y_j$ with $c-i+1\leq j\leq c-1$ are
  $y_{j+1},\ldots,y_{c-1},y_c,\ldots,y_n$. However, since
  $y_{j+1},\ldots,y_{c-1}$ are already present in the ideal and $y_j$ is
  multiplicative for them, it is sufficient to consider multiplication by
  $y_c,\ldots,y_n$. We now show that every term involving these variables
  that is multiplied by some $y_j$ has degree at most $a_c-1$. Suppose, for
  contradiction, that $y_c^{b_c}\cdots y_n^{b_n}$ of degree at least $a_c$
  is multiplied by $y_j$. Since $\mathcal{I}$ is strongly stable and
  $y_c^{a_c}\in\varphi(\mathcal{I})$, we deduce that the term
  $y^\gamma:=y_k^{c_k}y_{k+1}^{b_{k+1}}\cdots y_n^{b_n}$ with $k\geq c$,
  $c_k+\sum_{t=k+1}^nb_t=a_c$ and $c_k\leq b_k$, belongs to
  $\varphi(\mathcal{I})$.  Because the variables $y_j,y_c,\ldots,y_k$ are
  multiplicative for $y^\gamma$, we have
  $y^\gamma\mid_P y_jy_c^{b_c}\cdots y_n^{b_n}$ and therefore this term
  will not appear in the Pommaret basis.  This contradiction proves the
  claim. Thus, the required terms have degree at most $a_c-1$.

  Next, we examine the elements of $\varphi(H)$.  Observe that after adding
  the variables $y_{c-i+1},\ldots,y_{c-1}$, only those elements of
  $\varphi(H)$ involving $y_1,\ldots,y_{c-i},y_c,\ldots,y_n$ remain of
  interest. Among these, terms of class at least $c$ need no further
  analysis. Consider a term
  $y^\mu:=y_k^{b_k}\cdots y_\ell^{b_\ell}y_c^{b_c}\cdots y_n^{b_n}$ with
  $1\leq k\leq\ell\leq c-i$. Its non-multiplicative variables are
  $y_{k+1},\ldots,y_n$. It is clear that $y^\mu=\varphi(x^\mu)$ with
  $x^\mu:=x_{k+i-1}^{b_k}\cdots x_{\ell+i-1}^{b_\ell}x_c^{b_c}\cdots
  x_n^{b_n}$. Since multiplication of $x^\mu$ by its non-multiplicative
  variables $x_{k+i},\ldots,x_n$ (or equivalently
  $y_{k+1},\ldots,y_{c-i},y_c,\ldots,y_n$) has already been considered in
  the computation of $H$, no new element will be added to
  $\varphi(H)$. Hence the Pommaret basis of
  $\varphi(\mathcal{I})+\langle y_{c-i+1},\ldots,y_{c-1}\rangle$ is
  finite. From the above arguments, the degree of the Pommaret basis of
  this ideal is equal to the maximum degree of the elements involving the
  variables $y_1,\ldots,y_{c-i},y_c,\ldots,y_n$ which correspond to
  $x_i,\ldots,x_n$.

  \textbf{(Case 2)} Suppose $c\leq i\leq n$. The proof in this case is
  similar to that of Lemma~\ref{ss-4}.  Indeed, the quasi-stability of
  $\mathcal{I}+\langle x_1,\ldots,x_{i-1}\rangle$ follows from its
  zero-dimensionality. Moreover, the generating set of this ideal clearly
  consists only of $x_1,\ldots,x_{i-1}$ and terms of class at least
  $i$. Thus, to compute the degree of its Pommaret basis, it suffices to
  examine the multiplication of the non-multiplicative variables
  $x_i,\ldots,x_n$ (with $i\geq c$) by $x_1,\ldots,x_{i-1}$. We claim that
  any such product has degree at most $a_i$. Suppose, for contradiction,
  that there exists $m:=x_jx_i^{\ell_i}\cdots x_n^{\ell_n}$ with
  $1\leq j\leq i-1$ of degree at least $a_i+1$. Since $\mathcal{I}$ is
  strongly stable and $x_i^{a_i}\in\mathcal{I}$, the term
  $x^\gamma:=x_k^{c_k}x_{k+1}^{\ell_{k+1}}\cdots x_n^{\ell_n}$ with
  $k\geq i$, $c_k+\sum_{t=k+1}^n\ell_t=a_i$, and $c_k\leq \ell_k$ belongs
  to $\mathcal{I}$. Because $x_j,x_i,\ldots,x_n$ are multiplicative for
  $x^\gamma$, we conclude that $x^\gamma\ |_P\ m$, so $m$ is not added to
  the Pommaret basis and thus, the resulting terms have degree at most
  $a_i$.

  We now claim that $\omega_i(\mathcal{I})$ can be computed using only
  those elements of $H$ of class at least $i$. In other words, the value of
  $\mathrm{d}_i$ on terms of class smaller than $i$ is smaller than its
  value on terms of class at least $i$. Let
  $x^\alpha:=x^{\beta} x_i^{\ell_i}\cdots x_n^{\ell_n}\in H$ where
  $x^{\beta}\in\K[x_1,\ldots,x_{i-1}]$ and
  $\mathrm{d}_i(x^\alpha)=\omega_i(\mathcal{I})$. We show that
  $\deg(x^\beta)=0$. Assume, to the contrary, that $\deg(x^\beta)>0$. Since
  $\mathcal{I}$ is strongly stable,
  $\frac{x^{\beta}}{x_j}x_i^{\ell_i+1}x_{i+1}^{\ell_{i+1}}\cdots
  x_n^{\ell_n}\in \mathcal{I}$ for some $j<i$ with $x_j\ |\ x^\beta$. This
  term is a multiple of $x^{\gamma}:=x^{\mu}x_i^{b_i}\cdots x_n^{b_n}\in H$
  where $x^{\mu}\ |\ x^{\beta}/x_j$, $b_i\leq \ell_i+1$ and
  $b_k\leq \ell_k$ for every $i+1\leq k\leq n$. If $b_i\leq\ell_i$, then
  $x^\gamma\ |\ x^\alpha$, contradicting the fact that $H$ is a basis; thus
  $b_i=\ell_i+1$. From the maximality of $\mathrm{d}_i(x^\alpha)$, we have
  $b_i+b_{i+1}+\cdots+b_n\leq\ell_i+\ell_{i+1}+\cdots+\ell_n$. Substituting
  $b_i=\ell_i+1$ gives $1+b_{i+1}+\cdots+b_n\leq \ell_{i+1}+\cdots+\ell_n$
  and thus $b_{i+1}+\cdots+b_n<\ell_{i+1}+\cdots+\ell_n$, so there exists
  $k>i$ with $b_k<\ell_k$. Since $\mathcal{I}$ is strongly stable,
  $x^{\mu}x_i^{b_i}\cdots x_k^{b_k}\cdots x_n^{b_n}\in\mathcal{I}$ and
  $b_i=\ell_i+1$, we have
  $x^{\mu}x_i^{\ell_i}\cdots x_k^{b_{k}+1}\cdots x_n^{b_n}\in
  \mathcal{I}$. This term divides $x^\alpha$, contradicting the fact that
  $H$ is a basis. Consequently, the computation of $\omega_i(\mathcal{I})$
  requires only those elements of $H$ whose class is at least $i$, which
  completes the proof.
\end{proof}

By applying Lemma~\ref{ss-3}, we conclude that the maximal degree among the
terms involving only $x_i,\ldots,x_n$ (for $c\le i\le n$) is precisely
$a_i$. Therefore, from the last part of the above proof, we obtain:

\begin{corollary}\label{ss-axial-15}
  $\mathrm{reg}\,\bigl(\mathcal{I}+\langle x_1,\ldots,x_{i-1}\rangle\bigr)=
  a_i=\omega_i(\mathcal{I})$ for $c\leq i\leq n$.
\end{corollary}

\begin{example}
  Consider the ideal
  $\mathcal{I} = \langle x_4^2,\, x_3^3,\, x_3^2x_4,\, x_1x_3^2,\,
  x_2x_3^2,\, x_1x_3x_4,\, x_2x_3x_4 \rangle$ of the polynomial ring
  $\mathbb{Q}[x_1, x_2, x_3, x_4]$.  Observe that the ideal
  \(\mathcal{I} + \langle x_1 \rangle\) is not quasi-stable, since no pure
  power of \(x_2\) appears in it. Now, we apply the
  aforementioned linear change of variables.  Since the ideal
  \(\mathcal{I}\) is strongly stable, we have \(c = 3\). Thus, the change
  of variables is given by
  $$x_1 \mapsto y_2,\quad x_2 \mapsto y_1,\quad
  x_3 \mapsto y_3,\quad x_4 \mapsto y_4.$$
  It follows that
  $\mathcal{I} + \langle x_1 \rangle = \phi(\mathcal{I}) + \langle y_2
  \rangle$, and the Pommaret basis of this ideal is
  $$\bigl\{\,y_4^2,\, y_3^3,\, y_3^2y_4,\, y_1y_3^2,\, y_1y_3y_4,\, y_2,\,
  y_2y_3,\, y_2y_4,\, y_2y_3^2,\, y_2y_3y_4\,\bigr\}.$$
  Clearly, the degree of this basis is \(3\). On the other hand, we have
  \(\omega_2(\mathcal{I}) = 3\), and therefore the desired equality holds.
\end{example}

Our next result provides a simple way to compute for any quasi-stable ideal
$\mathcal{I}$ the regularities of all the ideals
$\mathcal{I}+\langle x_1,\ldots,x_{i-1}\rangle$ using the Pommaret basis of
$\mathcal{I}$.

\begin{proposition}\label{omega-2}
  Let \(\mathcal{I} \subset \mathcal{R}\) be a quasi-stable monomial ideal
  and let \(H\) be its Pommaret basis. Then,
  $\mathrm{reg}\,\bigl(\mathcal{I} + \langle x_1, \ldots, x_{i-1}
  \rangle\bigr) = \max {\bigl\{ \deg{(h)} \mid h \in H,\ \cls{(h)} \geq i
    \bigr\}}$.
\end{proposition}

\begin{proof}
  If we introduce for $i=1,\dots,n$ the ideals
  $\I^{(i-1)}=\I+\langle x_1,\ldots,x_{i-1}\rangle\subset\mathcal{R}$ and
  $\bar{\I}^{(i-1)}=\I|_{x_1=\cdots=x_{k-1}=0}\subset
  \mathcal{R}^{(i-1)}=\K[x_{i},\dots,x_{n}]$, then it is easy to see that
  we have the ring isomorphism
  $\mathcal{R}/\I^{(i-1)}\cong\mathcal{R}^{(i-1)}/\bar{\I}^{(i-1)}$.  One also
  easily verifies that
  $\bar{H}^{(i-1)}=\bigl\{h|_{x_1=\cdots=x_{k-1}=0}\mid h \in H,\ \cls{(h)}
  \geq i \bigr\}$ is a Pommaret basis of $\bar{\I}^{(i-1)}$ (see also
  \cite[Lemma~3.1]{zbMATH06126804}).  Now the claim follows immediately
  from Theorem~\ref{regpom}.
\end{proof}

\begin{example}
  Consider the ideal
  $\I=\langle x_4^2,x_3x_4,x_3^2,x_1x_4,x_2^2x_4\rangle\subset
  \mathbb{Q}[x_1,x_2,x_3,x_4]$.  It is easy to verify that this ideal is
  quasi-stable and that its Pommaret basis is
  $H=\{x_4^2,x_3x_4,x_3^2,x_1x_4,x_2^2x_4,x_1x_2x_4\}$.  We now want to
  compute the regularity of $\I^{(1)}=\I+\langle x_1\rangle$.  Following
  the proof of Proposition~\ref{omega-2}, we consider the ideal
  $\bar{\mathcal{I}}^{(1)} = \langle x_4^2,x_3x_4,x_3^2,x_2^2x_4\rangle
  \subset \mathbb{Q}[x_2,x_3,x_4]$.  From its Pommaret basis
  $\bar{H}^{(1)}=\{x_4^2,x_3x_4,x_3^2,x_2^2x_4\}$, we can immediately read
  off that
  $\mathrm{reg}(\bar{\mathcal{I}}^{(1)})=\mathrm{reg}(\I^{(1)})=3$.

  Alternatively, one could work with the coordinate transformation from the
  proof of Proposition~\ref{omega-1}.  Since there is no $j\in\mathbb{N}$
  such that $x_2^j\in \I$, the ideal $\I^{(1)}$ is not quasi-stable.
  Noting that here $D=2$ (as $H$ contains only power of $x_{3}$ and
  $x_{4}$), we apply the linear transformation $\varphi$ defined by
  $x_4\mapsto y_4,\ x_3\mapsto y_3,\ x_2\mapsto y_1,\ x_1\mapsto y_2$ and
  then consider the ideal $\varphi(\mathcal{I})+\langle y_2\rangle$.  This
  ideal is quasi-stable and from its Pommaret basis
  $H'=\{y_2,y_2y_3,y_2y_4,y_3^2,y_3y_4,y_4^2,y_1^2y_4\}$ we obtain
  $\mathrm{reg}(\varphi(\mathcal{I})+\langle y_2\rangle)=3$ confirming the
  result above.  Obviously, this approach is computationally more
  expensive, as it requires the computation of $H'$; using
  Proposition~\ref{omega-2} is much more direct.
\end{example}
  
%%% Local Variables:
%%% mode: latex
%%% TeX-master: "MainPaperFile"
%%% End:

%% file: B-Max.tex
\section{$\beta$-maximal and GIN position}
\label{sec:5}

This section provides a concise overview of the $\beta$-maximal position
introduced in \citep{HSS}. It reviews some key properties of this concept
and its application to the computation of generic annihilator numbers. In
addition, we present a novel deterministic algorithm for putting an ideal
into GIN position.

\begin{definition}
  Given a homogeneous polynomial ideal $\mathcal{I}\subset \mathcal{R}$, we
  refer to the finite set
  $B_q(\mathcal{I}) = \langle\lt{(\mathcal{I})}\rangle_q\cap\mathcal{T}$ as
  the \emph{$\K$-linear term basis} of $\langle\lt{(\mathcal{I})}\rangle_q$
  for any degree $q$ such that $\mathcal{I}_q\neq 0$. The
  \emph{$\beta$-numbers} of $\mathcal{I}$ are then given by the
  cardinalities
  $$\beta^{(i)}_q (\mathcal{I}) =\#\bigl\{t\in B_q(\mathcal{I})\ |\
  \mathrm{cls}(t) = i\bigr\}$$ and the \emph{$\beta$-vector} of
  $\mathcal{I}$ at degree $q$ is the ordered tuple
  $\beta_q(\mathcal{I})  = \big(\beta_q^{(n)}(\mathcal{I}),\ldots,
  \beta_q^{(1)}(\mathcal{I})\big) \in \mathbb{N}_0^n$.
  If $\mathcal{I}$ is in quasi-stable position and has the Pommaret basis
  $H$, we further define the \emph{$\beta$-numbers} of $H$ at degree $q$ by
  $\beta_q^{(k)}(H)=\#\{h\in H\mid \deg(h)=q,\,\cls(h)=k\}$ for any
  $1 \le k \le n$.
\end{definition}

As discussed in \cite[Remark~4.9]{HSS}, $\beta$-vectors are a useful tool
to study the asymptotic behaviour of Hilbert polynomials; in fact, knowing
the $\beta$-vector at a sufficiently high degree $q$ is equivalent to
knowing the Hilbert polynomial \cite[Prop.~8.2.6]{S}. Moreover, all
$\beta$-numbers of the ideal $\mathcal{I}$ can be computed recursively from
those of the Pommaret basis $H$ implying that the latter ones encode all
relevant information.

\begin{lemma}
  Let $\mathcal{I}\subset\mathcal{R}$ be a quasi-stable monomial ideal, and
  $H$ its Pommaret basis. Then we have for all degrees $q\geq1$ and all
  $1\leq k\leq n$ the equality
  $$\beta_q^{(k)}(\mathcal{I}) =
  \beta_q^{(k)}(H)+\sum_{j=k}^n\beta_{q-1}^{(j)}(\mathcal{I}).$$ 
\end{lemma}
\begin{proof}
  Consider a term $x^\alpha \in \mathcal{I}$ of degree $q$ with
  $\cls{(x^\alpha)} = k$, which contributes to
  $\beta_q^{(k)}(\mathcal{I})$. We distinguish two cases based on whether
  $x^\alpha$ belongs to $H$: If $x^\alpha \in H$, then $x^\alpha$
  contributes to $\beta_q^{(k)}(H)$. Otherwise, $x^\alpha \notin H$. Then
  $x^\alpha = x_k \cdot x^\beta$ for some $x^\beta \in \mathcal{I}$ of
  degree $q-1$ with $\cls{(x^\beta)} \geq k$.  The map
  $x^\beta \mapsto x_k \cdot x^\beta$ establishes a bijection between the
  terms in $\mathcal{I}$ of degree $q-1$ with class at least $k$ and the
  terms in $\mathcal{I} \setminus H$ of degree $q$ with class $k$. Thus,
  the number of terms in $\mathcal{I} \setminus H$ of degree $q$ with class
  $k$ equals $\sum_{j=k}^n \beta_{q-1}^{(j)}(\mathcal{I})$, yielding the
  desired equality.
\end{proof}

\begin{remark}\label{beta-alpha-1}
  Let $\mathcal{I}\subset\mathcal{R}$ be a homogeneous ideal and $H$ its
  Pommaret basis. Then, Theorem~\ref{alpha-pom.} states that
  $\alpha_{ij}(x_1,\ldots,x_n;\mathcal{R}/\mathcal{I})=\beta_{j+1}^{(i+1)}(H)$
  for all admissible values of the indices $i,j$.  Thus, if one displays
  the $\beta$-numbers and these annihilator numbers in tables, then the
  $\alpha$-table is up to an index shift simply the transpose of the
  $\beta$-table.  The annihilator numbers
  $\alpha_{ij}(x_1,\dots,x_n;\mathcal{R}/\mathcal{I})$ encode the degrees
  and classes of the generators in the Pommaret basis $H$ of $\mathcal{I}$.
  As these values can change under linear changes of variables, the
  $\beta$-numbers are generally not invariant.  This motivates the
  following definition of a generic position based on the $\beta$-vectors.
\end{remark}

\begin{definition}
  The ideal $\mathcal{I}$ is in \emph{$\beta$-maximal position}, if we have
  for any non-singular matrix $A\in K^{n\times n}$ and any degree $q$ with
  $\mathcal{I}_q\neq0$ the inequality
  \begin{equation}\label{bmaxlex}
    \beta_q(\mathcal{I})\succeq_{\mathrm{lex}}\beta_q(A\cdot\mathcal{I}).
  \end{equation}
\end{definition}

By \cite[Prop.~4.11]{HSS}, a homogeneous polynomial ideal
$\mathcal{I} \subset \mathcal{R}$ is in quasi-stable position, if and only
if the inequality \eqref{bmaxlex} holds for every non-singular matrix $A$
and all degrees $q \geq {\rm reg}\bigl(\lt(\mathcal{I})\bigr)$. Hence
$\beta$-maximal position implies quasi-stable position. A deterministic
test for \(\beta\)-maximal position was proposed in \cite{HSS}: one first
transforms the ideal into quasi-stable position (with a method developed in
that article) and then compares the $\beta$-vectors up to the degree
$\mathrm{reg}(\mathcal{I})$ with those of the generic initial ideal
(deterministically computed via Gr\"obner systems).  This approach is
extremely expensive and yields only a test, but not a method for achieving
$\beta$-maximal position, if the test fails.  We now present a much cheaper
method using solely linear algebra techniques to achieve either
$\beta$-maximal or GIN position assuming knowledge of the regularity
of~$\I$.  Starting from quasi-stable position is thus particularly useful,
since the regularity can then be read off directly from a Pommaret basis by
Theorem~\ref{regpom}.

\begin{proposition}[{\cite[Prop.~4.13]{HSS}}]\label{prop:betagin}
  If the ideal $\I$ is in GIN position, then it is also in $\beta$-maximal
  position.  In particular, if $\I$  is  in $\beta$-maximal
  position, then $\beta_{q}(\I)=\beta_{q}\bigl(\operatorname{Gin}(\mathcal
  I)\bigr)$ for all degrees~$q$ with $\I_{q}\neq0$.
\end{proposition}

\begin{remark}\label{genbmax}
  As a trivial consequence of this result, we note that the annihilator
  numbers $\alpha_{ij}(x_{1},\dots,x_{n};\mathcal{R}/\I)$ are the generic
  ones, if and only if the ideal $\I$ is in $\beta$-maximal position.
\end{remark}

Following \cite{bermejo}, to obtain $\operatorname{Gin}(\mathcal I)$, it
suffices to consider a linear change of coordinates whose associated matrix
$A=(a_{ij})$ is upper unitriangular; that is,
\begin{equation}\label{eq:generic-triangular-change}
  A\cdot x_i
  =x_i + a_{i,i+1} x_{i+1} + \cdots + a_{i,n} x_n,
  \qquad 1\leq i\leq n.
\end{equation}
It is further proved there that a non-empty Zariski open subset
$\mathcal V\subset \K^{\binom n2}$ exists such that
$\lt{\bigl(A\cdot \mathcal I\bigr)} =\operatorname{Gin}(\mathcal I)$ for
every $A \in \mathcal V$. Thus, a generic choice of $A$ yields the generic
initial ideal. A random choice, however, does not by itself certify that
$A$ belongs to the required open set.

Algorithm~\ref{alg:bmax} determines polynomial equations defining the
complement of $\mathcal V$.  It then suffices to choose coefficients which
do not satisfy these equations.  The algorithm applies a generic, upper
unitriangular, linear transformation of the variables leading to parametric
polynomials in $\bigl(\K[\mathbf{a}]\bigr)[x_1,\dots,x_n]$ where we
abbreviate $\K[\mathbf{a}]=\K[a_{1,2},\dots,a_{n-1,n}]$.  It then sets up
parametric $\mathcal{K}$-linear generating sets for the homogeneous
components $\mathcal{I}_q$ starting with the minimal degree~$q$ such that
$\mathcal{I}_{q}\neq0$ and going up to degree
$r=\mathrm{reg}(\mathcal{I})$.  From each of these generating sets, a
coefficient matrix $M_q$ is extracted.  Thus the columns of the matrix
$M_q$ are labelled by the terms of degree~$q$ sorted descendingly
by~$\prec_{drl}$.  Then, starting from the left, the algorithm eliminates
any column which does not increase the rank and thus produces in the end a
reduced matrix $\widehat{M}_{q}$ of full column rank consisting of columns
which have been in $M_q$ as far left as possible.  Finally, it computes a
list $C_q$ of all non-zero maximal minors of the reduced
matrix~$\widehat{M}_{q}$.

\begin{algorithm}[h]
\caption{Determine conditions characterising GIN position}
\label{alg:bmax}
\begin{algorithmic}[1]
\Require homogeneous generating set $G$ of the ideal
$\mathcal{I}\subset\mathcal{R}$, regularity $r=\mathrm{reg}(\I)$
\Ensure list of lists of polynomials in the parameters $a_{ij}$
with $1\leq i<j\leq n$ such that triangular linear transformations to GIN position are
determined by those parameter values for which in each list at least one
polynomial does not vanish
\State Transform $G$ by a generic triangular linear change of
variables $$x_i\mapsto x_i + a_{i,i+1} x_{i+1} + \cdots + a_{i,n} x_n,\quad
1\leq i\leq n,$$ with undetermined parameters $a_{ij}$ into a new set
$\tilde{G}$ of homogeneous polynomials with coefficients in the polynomial
ring $\mathcal{K}[\mathbf{a}]$.
\ForAll{$q_{1}=\min_{g\in G}{\deg{(g)}}\leq q\leq r$}
\State Construct all products $t\cdot\tilde{g}$ with
$\tilde{g}\in\tilde{G}$ and $t$ a term such that
$\deg{(t\cdot\tilde{g})}=q$. 
\State Order each product $t\cdot\tilde{g}$ by $\prec_{drl}$ and write its
coefficients into one row of the matrix $M_q$ over
$\mathcal{K}[\mathbf{a}]$. 
\State Initialise $\widehat{M}_{q}$ as empty matrix.
\For{$i=1,\dots,\#(\text{columns of }M_q)$}
\State Check whether or not attaching the $i$th column of $M_q$ to
$\widehat{M}_{q}$ increases the (column) rank of $\widehat{M}_{q}$ over
$\mathcal{K}(\mathbf{a})$. 
\If{yes}
\State Attach the column to $\widehat{M}_{q}$.
\EndIf
\EndFor
\State Compute list $C_q$ of all non-zero maximal minors of
$\widehat{M}_{q}$. 
\EndFor
\Return $\bigl[ C_{q_1},\dots,C_{r}\bigr]$
\end{algorithmic}  
\end{algorithm}

\begin{proposition}\label{ginpos}
  Any specialisation of the generic linear transformation where the
  parameters $a_{ij}$ are chosen such that in each list $C_q$ computed by
  Algorithm~\ref{alg:bmax} at least one maximal minor is not zero leads to
  an ideal in GIN position. This condition is satisfied for almost all
  specialisations.
\end{proposition}

\begin{proof}
  The GIN position (for $\prec_{drl}$) is characterised by the property
  that in each degree~$q$ with $\I_q\neq0$ the terms in $\lt{(\I)}_q$ are
  as large as possible for $\prec_{drl}$.  As the generic initial ideal is
  minimally generated in degrees up to the regularity~$r$, it suffices to
  verify this property only for these finitely many degrees.  The loop in
  the lines 8--11 determines for each such $q$ which terms appear in
  $\operatorname{Gin}(\mathcal I)_q$.  By construction, the same terms will
  be the leading terms of degree $q$ of the transformed ideal, if and only
  if at least one maximal minor in the list $C_q$ is non zero.
\end{proof}

If one wants a $\beta$-maximal position which is \emph{not} the GIN
position (which need not exist), then it is easier to work with Gaussian
elimination than with maximal minors.  The columns of the matrices $M_q$
can naturally be combined into groups according to the classes of the terms
labelling them.  Applying a (fraction-free) Gaussian elimination to $M_q$
using only row operations yields a row echelon form where the pivots sit in
those columns which Algorithm~\ref{alg:bmax} chooses in the construction of
the reduced matrix~$\widehat{M}_{q}$ and choosing a specialisation
satisfying the condition in Proposition~\ref{ginpos} will keep all these
pivots non zero.  For $\beta$-maximal position, pivots might be further to
the right, as long as they are still in a column with the same class.
Indeed, by Proposition~\ref{prop:betagin}, if an ideal~$\mathcal I$ is in
$\beta$-maximal position, then one has
\( \beta_q(\mathcal I)\;=\;\beta_q\bigl(\operatorname{Gin}(\mathcal
I)\bigr) \) for every degree $q\ge 0$ with $\mathcal I_q\neq 0$.  Thus, if
a specialisation can be chosen so that in every relevant degree the pivot
columns may move while preserving their classes, but at least one pivot
differs from the generic one, then the resulting ideal is in
\(\beta\)-maximal position but not in GIN position.

% The conditions computed by Algorithm~\ref{alg:bmax} completely characterise
% all specialisations leading to GIN position as a Zariski open set.  The
% prize is that for larger values of $q$, typically the number of maximal
% minors which have to be considered explodes (and sometimes most of them are
% irrelevant, if just one constant maximal minor exists).  By contrast, using
% Gaussian elimination is much cheaper and also yields a Zariski open set
% describing specialisations which all lead to GIN position.  However,
% generally there exist further specialisations also leading to GIN position
% and at least in low degrees the conditions derived via maximal minors are
% often simpler and weaker.

\begin{example}
We consider the ideal
\[
\mathcal{I}=\left\langle
x_3^3,\;
x_2x_3^2+x_1^2x_2,\;
x_2^3x_3,\;
x_2^4,\;
x_1x_2^2x_3,\;
x_1^2x_3^2,\;
x_1^4x_3
\right\rangle
\subset \QQ(\sqrt{3})[x_1,x_2,x_3].
\]
It will later turn out that it is already in $\beta$-maximal position, but
this cannot be easily verified.  By contrast, it is straightforward to
check that the ideal is in quasi-stable position (in fact even in stable
position), as $\lt{(\I)}$ is minimally generated by
{\setlength{\arraycolsep}{2pt}
\[
\lt{(\I)}=\langle x_3^3,\ x_2x_3^2,\ x_2^3x_3,\ x_2^4,\ x_1^2x_3^2,\
x_1x_2^2x_3,\ x_1^2x_2x_3,\ x_1^2x_2^3,\ x_1^3x_2^2,\ x_1^4x_2,\ x_1^4x_3 
\,\rangle
\]}%
and this minimal basis is also the Pommaret basis of $\lt{(\I)}$.  Thus the
first step of Algorithm~\ref{alg:bmax} is not necessary.  We also see that
here $q_1=3$ and $q_2=5$.  We apply now the generic transformation
$x_1\mapsto x_1+Ax_2+Bx_3$, $x_2\mapsto x_2+Cx_3$, $x_3\mapsto x_3$.

Since $M_4$ and $M_5$ are of size $10\times 15$ and $22\times 21$,
respectively, we explicitly write down only the matrix
\[\setlength{\arraycolsep}{2pt}
M_3=\begin{pmatrix}
1&0&0&0&0&0&0&0&0&0\\
C(1+B^2)&1+B^2+2ABC&A^2C+2AB&A^2&2BC&
2(AC+B)&2A&C&1&0
\end{pmatrix}.
\]
The reduction $\widehat{M}_3$ is simply given by the first two columns and
the only maximal minor yields the condition $1+B^2+2ABC\neq0$.  Obviously,
a Gaussian elimination simply eliminates the first entry of the second row
and leads to the same condition as pivot.

A (\textsc{Maple}) computation shows that the matrix $M_4$ also has full
row rank and that its reduction $\widehat{M}_4$ consists of its first 10
columns.  Thus we get again only one maximal minor leading to the condition
\[
-2A^2\bigl(6A^3C^3-12A^2BC^2+9AB^2C+3AC-3B^3-B\bigr)\neq0.
\]
A Gaussian elimination would lead here to 10 non-trivial pivots.  Their
form depends on the precise steps done during the elimination.  In our
\textsc{Maple} computation, we obtained for example the condition $C\neq0$,
whereas the above minor shows that the specialisation $A=B=1$ and $C=0$
satisfies all conditions in degrees 3 and 4.

In degree 5, it turns out that the rank of $M_5$ is 20 and thus in
principle one has to compute $\tbinom{22}{20}=231$ maximal minors as
determinants of $20\times20$-matrices.  However, it turns out that by
dropping the last column and the rows\footnote{Of course, the order in
  which the rows are written is arbitrary and the given indices just refer
  to the one specific choice we did in our computations.} 18 and 22 one
obtains a constant maximal minor.  Thus, this last degree imposes no
additional conditions and the specialisation $A=B=1$, $C=0$ yields GIN
position.  By contrast, a Gaussian elimination produced 20 complicated
polynomial conditions (many of which factor).

A comparison of the $\beta$-vectors now shows that the original variables,
corresponding to the specialisation $A=B=C=0$, are $\beta$-maximal.  In
degree~3, it is easy to see that even the condition for GIN position is
satisfied. In degree~4, obviously the minor vanishes, but by using the 11th
column instead of the 9th one obtains a new non-vanishing minor.  As this
change concerns only leading terms of class~1, the ideal is still in
$\beta$-maximal position, but the leading ideal is no longer the generic
initial ideal.

For constructing further $\beta$-maximal positions not leading to the
generic initial ideal, we note that the columns 2--4 of the matrix $M_3$
correspond to terms of class~2. For avoiding GIN position, we must choose
the parameters such that $1+B^2+2ABC=0$. This happens for example for
$A=B=-C=1$ and for this choice, the next entry in the second row of $M_3$
remains non-zero. Thus in degree~3 the leading ideal contains the term
$x_2^2x_3$ instead of $x_2x_3^2$ as $\operatorname{Gin}(\mathcal I)$.  In
degree~4, a check reveals that the same leading terms occur as in
$\operatorname{Gin}(\mathcal I)$.  For the choice $A=1$, $B=1/\sqrt{3}$,
$C=-2/\sqrt{3}$, the second and the third entry in the second row of $M_3$
vanish, whereas the fourth entry remains non-zero so that the leading ideal
contains now the term $x_2^3$. Again a check shows that in degree~4,
nothing happens. A complete examination of the corresponding pivot conditions shows that these are the only four \(\beta\)-maximal leading ideals arising in this example.
\end{example}

The complexity of Algorithm~\ref{alg:bmax} explodes rapidly when the
regularity $r$ grows, because of the large number of maximal minors that
exist in higher degrees.  The following result shows that the approach via
Gaussian elimination -- where only one matrix in each degree is studied --
has a worst complexity given essentially by $O(n^{4r})$; thus, it is
polynomial in $n$ and exponential in $r$.\footnote{One could also express
  the complexity as $O(r^{4n+2})$ reverting the roles of $n$ and $r$, but
  for typical ideals in algebraic geometry, it makes more sense to take $r$
  as the exponent.}  Nevertheless, the bound shows that deterministically putting an ideal into GIN position remains computationally expensive. 
\begin{proposition}
  Let $\mathcal I\subset\mathcal R=\K[x_1,\ldots,x_n]$ be a homogeneous
ideal given by a minimal homogeneous generating set $G$ and $r=\operatorname{reg}(\mathcal I)$. Then a
  non-empty Zariski-open set of parameter values $a_{ij}$ for which the
  corresponding unitriangular linear coordinate change puts $\mathcal I$ in
  GIN position can be determined using
  \[
    O\left(
      r^2\binom{n+r-1}{r}^{\,4}
    \right)\subseteq O\bigl(r^{2}n^{4r}\bigr)
  \]
  arithmetic operations in the rational function field $F=\K(\mathbf{a})$.
\end{proposition}
\begin{proof}
 For every degree $1\leq q\leq r$, set
  \[
    s_q=\dim_\K(\mathcal R_q)=\binom{n+q-1}{q},
    \qquad
    d_q=\dim_\K(\mathcal I_q),
  \]
  and let $b_q$ denote the number of rows of the matrix $M_q$ appearing in
  Algorithm~\ref{alg:bmax}. Thus,
  \[
    b_q=
    \sum_{\substack{g\in G\\ \deg(g)\leq q}}
    \binom{n+q-\deg(g)-1}{q-\deg(g)}.
  \]
  Consequently, $M_q$ is a $b_q\times s_q$ matrix of rank $d_q$ over
  $F$. We follow essentially Algorithm~\ref{alg:bmax}. However, instead of
  first constructing $\widehat M_q$ and subsequently computing all its
  maximal minors, we perform Gaussian elimination incrementally. More
  precisely, let $E_q$ be an auxiliary matrix, initially empty, which is
  maintained in column-echelon form. The columns of $M_q$ are considered
  from left to right. For each column $v$, we reduce $v$ with respect to
  the pivot columns already stored in $E_q$. If the reduced column is zero,
  then $v$ is discarded. Otherwise, its first non-zero entry determines a
  new pivot, the reduced column is appended to $E_q$, and the corresponding
  original column of $M_q$ is retained.

This way, $E_q$ contains both the pivot positions and the pivot expressions
arising during Gaussian elimination. Since the columns are processed in
their given order, these pivots determine generic column-rank conditions
for $M_q$.  After clearing the denominators of the pivots, their numerators
yield the required polynomial conditions in $\K[\mathbf{a}]$ and there is
no need to compute the complete collection of maximal minors of
$\widehat M_q$.

At most $d_q$ pivot columns are selected. Reducing one column of $M_q$ with
respect to these requires $O(b_qd_q)$ arithmetic operations in $F$. Since
$M_q$ has $s_q$~columns, the incremental Gaussian elimination in degree $q$
requires $O(b_qs_qd_q)$ arithmetic operations. Therefore, the total
complexity over all degrees up to the regularity is
\[
O\left(
\sum_{q=1}^{r}b_qs_qd_q
\right).
\]

We now bound this expression in terms of $n$ and $r$. Put $s=s_r$. For
every $q\leq r$, we have $s_q\leq s$ and $d_q\leq s_q\leq s$. For every $d$, let $G_d$ denote the set of elements of $G$ of degree
$d$. Since $G$ is a minimal homogeneous generating set, its elements of
degree $d$ are linearly independent over $\K$, and hence
\[
\#G_d\leq \dim_\K(\mathcal I_d)\leq
s_d=\binom{n+d-1}{d}.
\]
The number of rows of $M_q$ is therefore
\[
b_q=
\sum_{d=1}^{q}
\#G_d\binom{n+q-d-1}{q-d}
\leq
\sum_{d=1}^{q}s_ds_{q-d}
\leq q\,s_q^2\leq q\,s^2.
\]
Consequently,
$$\sum_{q=1}^{r}b_qs_qd_q\leq
\sum_{q=1}^{r}q\,s^2\cdot s\cdot s=
s^4\sum_{q=1}^{r}q=
\frac{r(r+1)}{2}s^4.$$
Thus, the claimed complexity follows.
\end{proof}

\begin{remark}
There is a distinction between the conditions obtained by
Algorithm~\ref{alg:bmax} and those produced in the proof of the
proposition.  For a fixed degree $q$, let $C_q$ denote the set of all
non-zero maximal minors of $\widehat M_q$.  Algorithm~\ref{alg:bmax}
describes the complete maximal-rank locus
\[
U_q
=
\K^N\setminus V(C_q),
\]
where $N=n(n-1)/2$.  Indeed, a specialization preserves the required
rank of $\widehat M_q$ if and only if at least one of the maximal
minors does not vanish.

By contrast, the incremental Gaussian elimination used in the proof
selects a particular sequence of pivots.  If
$p_{q,1},\ldots,p_{q,d_q}$ denote the corresponding pivot conditions
after clearing denominators, then it yields the open set
\[
U'_q
=
\K^N\setminus V\left(\prod_{j=1}^{d_q}p_{q,j}\right).
\]
In general, $U'_q\subseteq U_q$ and the inclusion may be strict.  In fact,
a pivot selected during Gaussian elimination may vanish at some
specialisation leading to GIN position, as another maximal minor of
$\widehat M_q$ remains non-zero.  Consequently, Algorithm~\ref{alg:bmax}
provides the full description of the Zariski-open set of admissible
specialisations, whereas the procedure in the proposition computes, at
lower cost, a possibly proper non-empty Zariski-open subset of it.
\end{remark}

%%% Local Variables:
%%% mode: latex
%%% TeX-master: "MainPaperFile"
%%% End:

%% file: Sreg.tex
\section{Sectional regularities and axial constants}
\label{sec:6}

In this section, we first recall the definitions of two recently introduced
families of invariants associated with a homogeneous ideal, namely the
sectional regularities and the axial constants. Typically, these invariants
are defined and computed using the generic initial ideal of a given
ideal. However, we will demonstrate that the sectional regularities of an
ideal can be easily obtained using a Pommaret basis, provided that the
ideal is in $\beta$-maximal position. Furthermore, we prove that when the
ideal is simultaneously in \(\beta\)-maximal and strongly stable position,
the sectional regularities and the axial constants coincide. Note that
again the first part of this section (up to Corollary~\ref{Cor5.8})
contains results that have already been presented in \citep{axial}. We will
provide here simpler and more constructive proofs of slightly more general
statements using the concepts of $\beta$-vectors and $\beta$-maximality.

Sectional regularities represent a very natural continuation of the
classical notion of Castelnuovo--Mumford regularity, as they arise by
taking hyperplane sections.  We slightly extend the definition given in
\citep{axial} by allowing also for non-generic sequences.

\begin{definition}
  Let $\mathcal{I}\subset\mathcal{R}$ be a homogeneous ideal and
  $\boldsymbol{\ell}=(\ell_1,\ldots,\ell_n)$ a sequence of $n$ linearly
  independent, linear forms.  The {\em $i$-th sectional regularity} of
  $\mathcal{I}$ with respect to the sequence $\boldsymbol{\ell}$ is defined
  by
  $$\mathrm{sreg}_i(\boldsymbol{\ell};\mathcal{I}) =
  \mathrm{reg}\,\bigl(\mathcal{I}+\langle\ell_{1},\ldots,\ell_{n-i}\rangle\bigr)$$
  for any $0\leq i\leq n$. In particular, for $i=n$, we have
  $\mathrm{sreg}_n(\boldsymbol{\ell};\mathcal{I})=\mathrm{reg}(\mathcal{I})$
  independently of the chosen sequence $\boldsymbol{\ell}$.
\end{definition}

We have seen in Proposition~\ref{fintab} that in the quasi-regular case the
annihilator tables are finite.  The sectional regularity
$\mathrm{sreg}_i(\boldsymbol{\ell};\mathcal{I})$ determines the minimal
degree where the last $i$ columns of the table are all zero.  The proof of
this observation, also for our slightly extended definition, is the same as
that of \cite[Cor.~3.7]{axial}.  As simple corollaries, one obtains some
further properties and characterisations of the sectional regularities.

\begin{proposition}\label{sreg-1} 
  Let $\mathcal{I}\subset\mathcal{R}$ be a homogeneous ideal and
  $\boldsymbol{\ell}$ a sequence of $n$ linearly independent, linear forms.
  \begin{enumerate}[(i)]
  \item The sequence $\boldsymbol{\ell}$ is quasi-regular for
$\mathcal R/\mathcal I$ if and only if, for every $1\leq i\leq n$,
    \[
      \mathrm{sreg}_i(\boldsymbol{\ell};\mathcal{I}) = \max{\bigl\{j\mid
        \alpha_{sj}(\boldsymbol{\ell};\mathcal{R/I}) \neq 0,\ s\geq
        n-i\bigr\}}+1.
    \]
    Moreover, $\mathrm{sreg}_0(\boldsymbol{\ell};\mathcal I)=1$.
  \item Let $H_{\boldsymbol{\ell}}$ be the Pommaret basis of the ideal
    $\mathcal{I}$ when using the linear forms $\ell_{i}$ in a quasi-regular
    sequence as new variables.  Then, for every $1\leq i\leq n$,
    \[
      \mathrm{sreg}_i(\boldsymbol{\ell};\mathcal{I}) = \max{\bigl\{j \mid
        \beta_{j+1}^{(s+1)}(H_{\boldsymbol{\ell}}) \neq 0,\ s \geq
        n-i\bigr\}}+1
    \]
    and $\mathrm{sreg}_0(\boldsymbol{\ell};\mathcal I)=1$.
  \item In the quasi-regular case, the sectional regularities form an
    ascending sequence:
    \[
      \mathrm{sreg}_0(\boldsymbol{\ell};\mathcal{I})\leq
      \mathrm{sreg}_1(\boldsymbol{\ell};\mathcal{I})\leq\cdots\leq
      \mathrm{sreg}_n(\boldsymbol{\ell};\mathcal{I}).
    \]
  \item Almost all quasi-regular sequences $\boldsymbol{\ell}$ lead
    to the same sectional regularities
    $\mathrm{sreg}_i(\boldsymbol{\ell};\mathcal{I})$, the \emph{generic
      sectional regularities}, which we denote by
    $\mathrm{sreg}_i(\mathcal{I})$.  They satisfy, for every $1\leq i\leq n$,
    \begin{align*}
      \mathrm{sreg}_i(\mathcal{I})
      &=\mathrm{reg}\,\bigl(\mathrm{Gin}(\mathcal{I})+\langle x_1,\ldots,x_{n-i}\rangle\bigr)\\
      &=\max{\bigl\{j \mid \beta_{j+1}^{(s+1)}(H_{\mathrm{Gin}(\I)})\neq 0,\ s\geq n-i\bigr\}}+1,
    \end{align*}
    where $H_{\mathrm{Gin}(\I)}$ denotes the Pommaret basis of
    $\mathrm{Gin}(\mathcal{I})$ (which is its minimal generating set when
    \(\mathrm{char}(\K) = 0\)).  In addition, $\mathrm{sreg}_0(\mathcal I)=1$.
  \end{enumerate}
 \end{proposition}

 \begin{proof}
   For part (i), we note that, by Proposition~\ref{fintab}, the maximum on
   the right hand side is finite, if and only if the sequence is
   quasi-regular, whereas the left hand side is always finite.  Hence
   quasi-regularity is a necessary condition.  In the quasi-regular case,
   introducing the modules
   $\mathcal{M}^{(i)} = \mathcal{R} /
   \langle\mathcal{I},\ell_{1},\dots,\ell_{n-i}\rangle$ for
   $0\leq i\leq n$, we have
   $\mathcal{M}^{(i-1)} \cong \mathcal{M}^{(i)} /
   \ell_{i}\mathcal{M}^{(i)}$.  An iteration over Proposition~\ref{regl}
   yields
   \begin{displaymath}
     \mathrm{reg}( \mathcal{M}^{(i)}) =
     \max{\bigl\{\mathrm{reg}(0\colon_{\mathcal{M}^{(j)}}\ell_{j})\mid j\leq
       i\bigr\}},
   \end{displaymath}
   since $\mathcal{M}^{(0)}\cong\K$.  Part (i) follows now immediately from
   the definition of the annihilator numbers.  Part (ii) is a trivial
   consequence of Theorem~\ref{alpha-pom.} and part (iii) of part (i).  The
   claims in part (iv) are simple consequences of (i), (ii) and the
   genericity statement in Theorem~\ref{genann}.
 \end{proof}

 \begin{example}\label{sregnbmax}
   We consider the ideal
   $\I=\langle x_3^2,\, x_1x_2+x_2x_3,\, x_1x_3,\ x_2^3,\, x_1x_2^2 \rangle
   \subset \QQ[x_1,x_2,x_3]$ with the term order
   $x_1 \prec_{{drl}} x_2 \prec_{{drl}} x_3$.  In \cite[Ex.~5.13]{HSS}, it
   was shown that its leading ideal
   $\LT{(\I)}=\langle x_3^2,\, x_2x_3,\, x_1x_3,\, x_2^3,\, x_1x_2^2,\,
   x_1^2x_2\rangle$ is strongly stable, but not in $\beta$-maximal position
   (and thus also not in GIN position).  We first compute the sectional
   regularities for the quasi-regular sequence
   $\boldsymbol{\ell}=(x_{1},x_{2},x_{3})$ defined by the chosen
   coordinates.  Since $\LT{(\I)}$ is strongly stable, its minimal
   generators are the leading terms of the Pommaret basis
   $H_{\boldsymbol{\ell}}$ of $\I$.  Thus the only non-vanishing
   $\beta$-numbers of $H_{\boldsymbol{\ell}}$ are contained in the two $\beta$-vectors
   $\beta_{2}(H_{\boldsymbol{\ell}})=(1,\,1,\,1)$ and
   $\beta_{3}(H_{\boldsymbol{\ell}})=(0,\,1,\,2)$.
   Theorem~\ref{alpha-pom.} then yields the following table for the
   annihilator numbers
   $\alpha_{ij}(\boldsymbol{\ell},\mathcal{R}/\mathcal{I})$:
   \begin{displaymath}
     \begin{array}{c|ccc}
       \alpha & 0 & 1 & 2 \\
       \hline
       0 & \cdot & 1 & 2 \\
       1 & \cdot & 1 & 1 \\
       2 & \cdot & 1 & \cdot
     \end{array}\quad.
   \end{displaymath}
   Proposition~\ref{sreg-1} (i) or (ii) yields for the sectional
   regularities $\mathrm{sreg}_{i}(\boldsymbol{\ell},\mathcal{I})$ the
   ascending sequence $(1,\,2,\,3,\,3)$.  For this simple ideal, they can
   also be easily verified by directly adjoining one variable after the
   other and computing modulo the adjoined variables.

   Using the method described in the last section, one finds that the
   generic initial ideal is given by
   $\mathrm{Gin}(\I)=\langle x_3^2,\, x_2x_3,\, x_2^2,\, x_1^2x_3,\,
   x_1^2x_2\rangle$ leading to the $\beta$-vectors
   $\beta_{2}(H_{\mathrm{Gin}(\I)})=(1,\,2,\,0)$ and
   $\beta_{3}(H_{\mathrm{Gin}(\I)})=(0,\,0,\,2)$ and the generic
   annihilator numbers
    \begin{displaymath}
     \begin{array}{c|ccc}
       \alpha & 0 & 1 & 2 \\
       \hline
       0 & \cdot & \cdot & 2 \\
       1 & \cdot & 2 & \cdot \\
       2 & \cdot & 1 & \cdot
     \end{array}\quad.
   \end{displaymath}
   Thus the generic sectional regularities are given by the sequence $(1,\,
   2,\, 2,\, 3)$.  Again, it is not difficult to verify this directly from
   the definition.
\end{example}

\begin{definition}
  Let $\mathcal{I}\subset\mathcal{R}$ be a homogeneous ideal and
  $\boldsymbol\ell=(\ell_1,\ldots,\ell_n)$ a quasi-regular sequence for
  $\mathcal{R}/\I$.  An annihilator number
  $\alpha_{ij}(\boldsymbol\ell;\mathcal{R/I})$ is called {\em co-extremal},
  if it is non-zero and $\alpha_{st}(\boldsymbol\ell;\mathcal{R/I})=0$ for
  every $(s,t)\neq (i,j)$ with $s\geq i, t\geq j$.
\end{definition}

Given a quasi-regular sequence $\boldsymbol{\ell}$, we denote again by
$H_{\boldsymbol{\ell}}$ the Pommaret basis of $\I$, in the coordinates
defined by $\boldsymbol{\ell}$.  By Theorem~\ref{alpha-pom.}, the
definition of co-extremal annihilator numbers is equivalent to the
condition that $\beta_{j+1}^{(i+1)}(H_{\boldsymbol{\ell}})\neq 0$, while
$\beta_{t+1}^{(s+1)}(H_{\boldsymbol{\ell}})=0$ for every $(s,t)\neq(i,j)$
with $s\geq i, t\geq j$.  Our generalised definition of sectional
regularities leads then to the following generalisation of
\cite[Cor.~3.9]{axial} (their generic statement is a trivial corollary of
our result).

\begin{proposition}\label{Cor5.8}
  Let $\mathcal{I\subset R}$ be a homogeneous ideal and
  $\boldsymbol\ell=(\ell_1,\ldots,\ell_n)$ a quasi-regular sequence of
  linear forms.  If $\alpha_{n-i,j}(\boldsymbol\ell;\mathcal{R/I})$ is a
  co-extremal annihilator number, then
  $\mathrm{sreg}_i(\boldsymbol{\ell};\mathcal{I})=j+1$.  Conversely, if
  $\mathrm{sreg}_{i-1}(\boldsymbol{\ell};\mathcal{I}) <
  \mathrm{sreg}_i(\boldsymbol{\ell};\mathcal{I})=j+1$, then
  $\alpha_{n-i,j}(\boldsymbol\ell;\mathcal{R/I})$ is a co-extremal
  annihilator number.  Consequently, the locations of the co-extremal
  annihilator numbers control both the jumps and the values attained by the
  sequence
  $\bigl(\mathrm{sreg}_{1}(\boldsymbol{\ell};\mathcal{I}), \ldots,
  \mathrm{sreg}_n(\boldsymbol{\ell};\mathcal{I})\bigr)$.
\end{proposition}

\begin{proof}
  Assume first that $\alpha_{n-i,j}(\boldsymbol\ell;\mathcal{R/I})$ is a
  co-extremal annihilator number.  Then,
  $\beta_{j+1}^{(n-i+1)}(H_{\boldsymbol{\ell}})\neq 0$ and for all
  $s\geq n-i, t\geq j$ with $(s,t)\neq (i,j)$, we have
  $\beta_{t+1}^{(s+1)}(H_{\boldsymbol{\ell}})=0$.  It follows from
  Proposition~\ref{sreg-1} that
  $$\mathrm{sreg}_i(\boldsymbol{\ell};\mathcal{I})
  =\max{\bigl\{t\mid \beta_{t+1}^{(s+1)}(H_{\boldsymbol{\ell}})\neq 0,\
    s\geq n-i\bigr\}}+1=j+1.$$ Now, suppose that
  $\mathrm{sreg}_{i-1}(\boldsymbol{\ell};\mathcal{I}) <
  \mathrm{sreg}_i(\boldsymbol{\ell};\mathcal{I})=j+1$.  Then,
  \begin{align*}
    \mathrm{sreg}_{i-1}(\boldsymbol{\ell};\mathcal{I})&=
         \max{\bigl\{t\mid\beta_{t+1}^{(s+1)}(H_{\boldsymbol{\ell}})\neq 0,\ s\geq n-i+1\bigr\}}+1\\
     &<\max{\bigl\{k\mid \beta_{k+1}^{(s+1)}(H_{\boldsymbol{\ell}})\neq
       0,\ s\geq n-i\bigr\}}+1 =
       \mathrm{sreg}_i(\boldsymbol{\ell};\mathcal{I})=j+1
  \end{align*}
  which implies that $\beta_{j+1}^{(n-i+1)}(H_{\boldsymbol{\ell}})\neq 0$
  and every $t$ for which $\beta_{t+1}^{(s+1)}(H_{\boldsymbol{\ell}})\neq 0$
  with $s\geq n-i+1$, is smaller than $j$.  Hence Theorem~\ref{alpha-pom.} yields
  $$j = \max{\bigl\{k \mid \alpha_{n-i,k}(\boldsymbol{\ell};\mathcal{R}/\mathcal{I})\neq 0\bigr\}} >
  \max{\bigl\{t \mid
    \alpha_{n-\ell,t}(\boldsymbol{\ell};\mathcal{R}/\mathcal{I})\neq 0,\
    \ell<i\bigr\}}.$$ Hence
  $\alpha_{n-i,j}(\boldsymbol{\ell};\mathcal{R}/\mathcal{I})\ne 0$ and
  $\alpha_{st}(\boldsymbol{\ell};\mathcal{R}/\mathcal{I})=0$ for all index
  pairs $(s,t)\neq (n-i,j)$ with $s\geq n-i, t\geq j$, which proves the
  co-extremality property.  The final statement is an obvious consequence
  of the first two statements.
\end{proof}

It is now a natural question under what conditions the simple sequence
$\boldsymbol{\ell}=(x_{1},\dots,x_{n})$ leads to the generic sectional
regularities.  In view of Remark~\ref{genbmax}, one expects that this
happens in $\beta$-maximal position and this is indeed the case.

\begin{theorem}\label{5.61-ths}
  Suppose that $\mathcal{I}\subset \mathcal{R}$ is a homogeneous ideal in
  $\beta$-maximal position, and let $H$ be its Pommaret basis. Then, for every $1\leq i\leq n$,
  \begin{align*}
    \mathrm{sreg}_i(\mathcal{I}) &=
         \max{\bigl\{j\mid \beta_{j+1}^{(s+1)}\bigl(\lt{(H)}\bigr) \neq 0,\
                                   s \geq n-i\bigr\}} + 1\\
    &= \mathrm{sreg}_{i}(x_{1},\dots,x_{n};\I) =
      \mathrm{reg}\,\bigl(\mathcal{I} + \langle x_1, \ldots, x_{n-i}
      \rangle\bigr). 
  \end{align*}
  Moreover, $\mathrm{sreg}_0(\mathcal I)=1$.
\end{theorem}

\begin{proof}
  According to Proposition \ref{sreg-1}(iv), we have
  $$\mathrm{sreg}_i(\mathcal{I}) =
  \max{\bigl\{j \mid \beta_{j+1}^{(s+1)}(H_{\mathrm{Gin}(\I)}) \neq 0,\ s
    \geq n-i\bigr\}} + 1.$$ By the assumption that $\mathcal{I}$ is in
  $\beta$-maximal position and by using Corollary~\ref{alpgaIlt} together
  with Theorem~\ref{genann}, the first equality follows for each $j$ with
  $\mathcal{I}_{j+1} \neq 0$.  Using this equality and
  Theorem~\ref{omega-2}, we obtain
  \begin{align*}
    \mathrm{sreg}_i(\mathcal{I}) &= \max{\{ j \mid \beta_{j+1}^{(s+1)}(\lt(H)) \neq 0,\ s \geq n - i \}} + 1\\
                                 &= \max{\{ \deg(h) \mid h \in \lt(H),\
                                   \cls(h) \geq n - i + 1 \}}\\
                                   &=
                                   \mathrm{reg}\bigl(\mathcal{I} + \langle x_1, \ldots, x_{n-i} \rangle\bigr).\qedhere
  \end{align*}
\end{proof}

In Example~\ref{sregnbmax}, we noted that the considered ideal was
\emph{not} in $\beta$-maximal position.  Hence it is not surprising that we
obtained non generic sectional regularities using the variables as
quasi-regular sequence.  In the following example, we consider an ideal
which is in $\beta$-maximal position, but not in GIN position.  We will see
that now one can use the given variables as linear forms for computing the
generic sectional regularities.

\begin{example}
  Consider the monomial ideal
  $\mathcal{I}=\langle x_4^2,x_3x_4,x_3^2,x_2^2x_4,x_1x_4\rangle\subset
  \mathbb{Q}[x_1,x_2,x_3,x_4]$ with the term order
  $x_1\prec_{drl}\cdots\prec_{drl}x_4$.  The Pommaret basis of $\I$ is
  obtained by adding the term $x_{1}x_{2}x_{4}$.  If we use
  $\boldsymbol{\ell}=(x_{1},x_{2},x_{3},x_{4})$, then a straightforward calculation yields
  \begin{gather*}
    \mathrm{sreg}_4(\boldsymbol{\ell};\mathcal{I})=
    \mathrm{sreg}_3(\boldsymbol{\ell};\mathcal{I})=
    \mathrm{reg}(\I)=3\\
    \mathrm{sreg}_2(\boldsymbol{\ell};\mathcal{I})=
    \mathrm{sreg}_1(\boldsymbol{\ell};\mathcal{I})= 2,\quad
    \mathrm{sreg}_0(\boldsymbol{\ell};\mathcal{I})= 1,
  \end{gather*}
  which also confirms
  $\mathrm{sreg}_0(\boldsymbol{\ell};\mathcal{I}) \leq\cdots\leq
  \mathrm{sreg}_4(\boldsymbol{\ell};\mathcal{I})$ as stated in
  Proposition~\ref{sreg-1}(iii).  Using the methods discussed above, we
  obtain
  $\mathrm{Gin}(\mathcal{I})=\langle
  x_2x_4,x_3^2,x_3x_4,x_4^2,x_1^2x_4\rangle$ and by comparing
  $\beta$-vectors one can easily verify that $\I$ is in $\beta$-maximal,
  but not in GIN position.  Furthermore, the $\beta$-vectors confirm
  Proposition~\ref{sreg-1}(ii).

  Let us now apply the results of Proposition~\ref{sreg-1} to the ideal
  $\mathcal{I}$ to obtain the generic sectional regularities.  For this
  purpose, we set $\ell_1=a_{1,4}x_4+a_{1,3}x_3+a_{1,2}x_2+x_1$,
  $\ell_2=a_{2,4}x_4+a_{2,3}x_3+x_2$, $\ell_3=a_{3,4}x_4+x_3$,
  $\ell_{4}=x_{4}$ with undetermined coefficients~$a_{ij}$.  We
  trivially have that
  $\mathrm{sreg}_0(\mathcal{I})=\mathrm{reg}(\mathcal{I}+\langle
  \ell_1,\ell_2,\ell_3,\ell_{4}\rangle)=1$.  Next, we determine
  $\mathrm{sreg}_1(\mathcal{I})=\mathrm{reg}(\mathcal{I}+\langle
  \ell_1,\ell_2,\ell_3\rangle)$.  By computing a Pommaret basis of the
  ideal $\mathcal{I}+\langle \ell_1,\ell_2,\ell_3\rangle$, we find that its
  regularity is~$2$.  On the other hand, we also have
  $\mathrm{reg}\bigl(\mathrm{Gin}(\mathcal{I})+\langle
  x_1,x_2,x_3\rangle\bigr)=2$ which confirms the equality of these two
  quantities.  Moreover, it is easy to verify that
  $\beta_{2}^{(4)}(H_{\mathrm{Gin}(\I)})=1$ and
  $\beta_3^{(4)}(H_{\mathrm{Gin}(\I)})=0$, so that
  $\max{\{j\ |\ \beta_{j+1}^{(s+1)}(H_{\mathrm{Gin}(\I)})\neq 0,\ s\geq
    3\}}+1=2$.  Repeating the same process for $i=2,3$ yields
  $\mathrm{sreg}_2(\mathcal{I})=\mathrm{sreg}_3(\mathcal{I})=2,
  \mathrm{sreg}_4(\mathcal{I})=\mathrm{reg}(\mathcal{I})=3$, i.\,e.\
  exactly the same values as we obtained above using the variables as
  linear forms as predicted by Theorem~\ref{5.61-ths}.
\end{example}

The notion of an axial constant was introduced in \cite{CH25a} and a generic
definition was again given subsequently by \citep{axial}. We adapt this definition to
our purposes by relating it to Pommaret bases.  It follows from well-known
properties of quasi-stable ideals (see e.\,g.\ \cite[Prop.~5.3.4]{S}) that
$a_i(\mathcal{I})<\infty$ only for $\dim{\mathcal{I}}<i\leq n$.

\begin{definition}
  Let $\mathcal{I}\subset\mathcal{R}$ be an ideal in quasi-stable position
  and $H$ its Pommaret basis.  The {\em $i$-th axial constant} of
  $\mathcal{I}$ is defined by
  $$a_i(\mathcal{I})=\min{\bigl\{j\in\mathbb{N}\ |\
    x_i^j\in\LT{(H)}\bigr\}}$$ for any $1\leq i\leq n$.  If no pure power
  of $x_i$ appears in $\LT{(H)}$, we set $a_i(\mathcal{I})=\infty$.
  Moreover, we call $a_i\bigl(\mathrm{Gin}(\mathcal{I})\bigr)$ the {\em
    $i$-th generic axial constant} of $\mathcal{I}$.
\end{definition}

\begin{proposition}
  If $\mathrm{char}(\K)=0$, then we have
  $\mathrm{sreg}_i(\mathcal{I})=a_{n-i+1}\bigl(\mathrm{Gin}(\mathcal{I})\bigr)$
  for every index $1\leq i\leq n-\mathrm{dim}(\mathcal{I})$.
\end{proposition}

\begin{proof}
  Since we assume that $\mathrm{char}(\K)=0$, the ideal
  $\mathrm{Gin}(\mathcal{I})$ is strongly stable.  Hence, by Corollary
  \ref{ss-axial-15}, we get
  $$\mathrm{sreg}_i(\mathcal{I})=
  \mathrm{reg}\bigl(\mathrm{Gin}(\mathcal{I})+\langle
  x_1,\ldots,x_{n-i}\rangle\bigr)=
  \omega_{n-i+1}\bigl(\mathrm{Gin}(\mathcal{I})\bigr)=
  a_{n-i+1}\bigl(\mathrm{Gin}(\mathcal{I})\bigr).\qedhere$$
\end{proof}

\begin{corollary}
  Let \(\mathcal{I} \subset \mathcal{R}\) be a homogeneous ideal that is
  both strongly stable and in $\beta$-maximal position. Then, for each
  \(1 \le i \le n-\mathrm{dim}(\mathcal{I})\), we have
  \( \mathrm{sreg}_i(\mathcal{I}) = a_{n - i + 1}(\mathcal{I})
  \). Furthermore, if \(\mathrm{char}(\K) = 0\), these quantities coincide
  with \(a_{n - i + 1}\bigl(\mathrm{Gin}(\mathcal{I})\bigr)\).
\end{corollary}

\begin{proof}
  Invoking the previous theorem and noting that \(\mathcal{I}\) is in
  \(\beta\)-maximal position, we deduce that
  \(\mathrm{sreg}_i(\mathcal{I}) = \mathrm{reg}\big(\mathcal{I} + \langle
  x_1, \ldots, x_{n-i} \rangle\big)\).  Moreover, since \(\mathcal{I}\) is
  in quasi-stable position, Corollary~\ref{alpgaIlt} implies that
  \(\mathrm{sreg}_i(\mathcal{I}) =
  \mathrm{sreg}_i\bigl(\lt(\mathcal{I})\bigr)\).  On the other hand,
  because \(\lt(\mathcal{I})\) is strongly stable, Proposition~\ref{omega-1}
  yields
  \(\mathrm{reg}\big(\lt(\mathcal{I}) + \langle x_1, \ldots, x_{n-i}
  \rangle\big) = \omega_{n - i + 1}\bigl(\lt(\mathcal{I})\bigr)\).
  Furthermore, by Corollary~\ref{ss-axial-15}, for every
  \(c \le n - i + 1 \le n\), we have
  \(\omega_{n - i + 1}\bigl(\lt(\mathcal{I})\bigr) = a_{n - i +
    1}(\mathcal{I})\).  The final equality follows from the fact that
  \(\mathrm{Gin}(\mathcal{I})\) is strongly stable when
  \(\mathrm{char}(\K) = 0\).
\end{proof}
%%%%%%%%%%%%%%%%%%%%%%%%%%%%%%%%%%%%%%%%%%%%%%%%%%%%%%%%%%%%%%

%%% Local Variables:
%%% mode: latex
%%% TeX-master: "MainPaperFile"
%%% End: